\documentclass[10pt]{article}
\usepackage{amssymb}
\usepackage{amsmath}
\usepackage{amsfonts}
\usepackage{multicol}
\usepackage[svgnames]{xcolor}
\usepackage{color}
\usepackage[plainpages=false,pdfpagelabels,bookmarksnumbered, colorlinks=true, linkcolor=sepia, citecolor=sepia, unicode]{hyperref}

\definecolor{sepia}{cmyk}{0, 0.83, 1, 0.70}
\newtheorem{theorem}{Theorem}[section]
\newtheorem{corollary}[theorem]{Corollary}
\newtheorem{lemma}[theorem]{Lemma}

\newtheorem{proposition}[theorem]{Proposition}
\newtheorem{remark}[theorem]{Remark}
\newtheorem{example}[theorem]{Example}
\newtheorem{definition}[theorem]{Definition}
\numberwithin{equation}{section}
\numberwithin{theorem}{section}
\newenvironment{proof}[1][Proof]{\noindent\textbf{#1.} }{\ \rule{0.5em}{0.5em}}
\input{tcilatex}
\begin{document}

\title{Weak, Strong and Linear Convergence of the CQ-Method Via  the
Regularity of Landweber Operators}
\author{Andrzej Cegielski, Simeon Reich and Rafa\l\ Zalas}
\maketitle

\begin{abstract}
We consider the split convex feasibility problem in a fixed point setting.  Motivated by  the
well-known CQ-method of Byrne (2002), we define an abstract Landweber
transform which applies to more general operators than the metric
projection. We call the result of this transform a Landweber operator. It
turns out that the Landweber transform preserves many interesting
properties. For example, the Landweber transform of a (quasi/firmly)
nonexpansive mapping is again (quasi/firmly) nonexpansive. Moreover, the
Landweber transform of a (weakly/linearly) regular mapping is again
(weakly/linearly) regular.  The preservation of regularity is important  because  it leads  to (weak/linear) convergence of many CQ-type methods. \medskip


\end{abstract}


\section{Introduction}

Let $\mathcal{H}_{1}$ and $\mathcal{H}_{2}$ be two real Hilbert spaces and
let $A:\mathcal{H}_{1}\rightarrow \mathcal{H}_{2}$ be a nonzero bounded
linear operator.

The \textit{split convex feasibility problem} (SCFP) is to
\begin{equation}
\text{find $x\in C$ \quad  such that  \quad $Ax\in Q$,}  \label{int:SCFP}
\end{equation}%
where $C\subseteq \mathcal{H}_{1}$ and $Q\subseteq \mathcal{H}_{2}$ are
nonempty, closed and convex. In this paper we assume that the SCFP has at
least one solution, that is, $C\cap A^{-1}(Q)\neq \emptyset $, and that
\begin{equation}
C:=\limfunc{Fix}S\quad \text{and}\quad Q:=\limfunc{Fix}T,
\label{int:SCFP-fix}
\end{equation}%
for some given operators $S$ and $T$. The SCFP was introduced by Censor and
Elfving \cite{CE94} for $\mathcal{H}_{1}=\mathbb{R}^{m}$ and $\mathcal{H}%
_{2}=\mathbb{R}^{n}$  and has attracted a lot of attention since then. Before
describing the contribution of our paper, we briefly recall  a  few results and methods which  have had  significant impact on this field.

\subsection*{Related work}

Among various methods designed for solving \eqref{int:SCFP}, the most
celebrated one is the \textit{$CQ$-method} of Byrne \cite{Byr02} defined by

\begin{equation}
x_{0}\in \mathcal{H}_{1};\quad x_{k+1}:=P_{C}\left( x_{k}+\frac{\lambda _{k}}{%
\Vert A\Vert ^{2}}A^{\ast }\Big(P_{Q}(Ax_{k})-Ax_{k}\Big)\right) ,\quad
k=0,1,2,\ldots ,  \label{int:CQmethod}
\end{equation}%
where $\lambda _{k}\in \lbrack \varepsilon ,2-\varepsilon ]$ for some $%
\varepsilon \in (0,1)$, $P_{C}$ and $P_{Q}$ are  the  metric projections onto $C$
and $Q$, respectively, and $A^{\ast }:\mathcal{H}_{2}\rightarrow \mathcal{H}_{1}$ is the adjoint operator to $A$. The above method was shown to converge to a minimizer of
\begin{equation}
f(x):=\frac{1}{2}\Vert P_{Q}(Ax)-Ax\Vert ^{2}  \label{int:f}
\end{equation}%
over $C$, assuming that such a minimizer exists. In the consistent case the
limit point becomes a member of $C\cap A^{-1}(Q)$. As it  has already been  mentioned by Byrne,  a  special case of the method \eqref{int:CQmethod}, with $C=\mathbb{R}^{m}$ and $Q=\{b\}\subset \mathbb{R}^{n}$, was introduced
by Landweber in \cite{Lan51}.  Therefore,  the CQ-method is sometimes referred to
as a \textit{projected Landweber method}; see, for example, \cite{PB97, JEKC06, ZC10} and \cite[Chapter 5]{Ceg12}.

 Because of  the differentiability of the squared distance function, the $CQ$%
-method can be viewed as the projected gradient method (PGM)
\begin{equation}
x_{0}\in \mathcal{H}_{1};\quad x_{k+1}:=P_{C}\left( x_{k}-\frac{\lambda _{k}}{%
L}\nabla f(x_{k})\right) ,\quad k=0,1,2,\ldots,   \label{int:ProjG}
\end{equation}%
with  a  convex and differentiable $f$,  where  $L=\Vert A\Vert ^{2}$  is  a Lipschitz constant of $\nabla f$. The convergence analysis of the PGM can be found, for example, in \cite[Chapter VII]{Pol87}  with  the constant parameter $\lambda _{k}=\lambda \in (0,2)$. Depending on the choice of the objective $f$ in \eqref{int:ProjG}, one can consider various extensions of the basic
CQ-method. For example, in the case of the \textit{multiple-set split
convex feasibility problem} (MSSCFP) \cite{CEKB05}, which is to
\begin{equation}
\text{find }x\in C\cap \bigcap_{i=1}^{m}C_{i}\quad \text{such that}\quad Ax\in
\bigcap_{j=1}^{n}Q_{j},  \label{int:MSSCFP}
\end{equation}%
the following objective
\begin{equation}
f(x):=\frac{1}{2}\sum_{i=1}^{m}\alpha _{i}\Vert P_{C_{i}}x-x\Vert ^{2}+\frac{%
1}{2}\sum_{j=1}^{n}\beta _{j}\Vert P_{Q_{j}}(Ax)-Ax\Vert ^{2},\quad \alpha
_{i},\beta _{j}>0,  \label{int:f-sim}
\end{equation}%
when combined with \eqref{int:ProjG}, led to a class of simultaneous
projection algorithms which weakly converge to a minimizer of $f$ over $C$.
The above $f$ has been introduced in \cite{CEKB05} in the Euclidean space
and further considered in \cite{Xu06} and \cite{MR07}  in a  general Hilbert space. It can be shown that $\nabla f$ is Lipschitz continuous with $%
L=\sum_{i=1}^{m}\alpha _{i}+\Vert A\Vert ^{2}\sum_{j=1}^{n}\beta _{j}$.

On the other hand, by the Lipschitz continuity of $\nabla f$ and with a
fixed value of $\lambda _{k}=\lambda \in (0,2)$, the PGM method %
\eqref{int:ProjG} becomes an example of the well-known Krasnosel'ski{\u{\i}}%
-Mann \cite{Man53, Kra55} method with an averaged operator. Thanks to the
above observation, the weak convergence of the $CQ$-method in a Hilbert
space was established, for example, in \cite{Byr04} and in \cite{Xu10}. See
also \cite{Xu11} for a weak convergence result with varying $\lambda _{k}\in
\lbrack \varepsilon ,2-\varepsilon ]$ in the PGM.

In some cases, it may happen that one has more information regarding the
sets $C$ and $Q$ which determine the SCFP \eqref{int:SCFP}. For example,
following Yang \cite{Yan04}, one could assume that
\begin{equation}
C:=\{x\in \mathcal{H}_{1}\mid c(x)\leq 0\}\quad \text{and}\quad Q:=\{y\in
\mathcal{H}_{2}\mid q(y)\leq 0\}  \label{int:SCFP-sublevel1}
\end{equation}%
for some lower semi-continuous convex functions $c\colon \mathcal{H}_{1}\rightarrow \mathbb{R}$ and $%
q\colon \mathcal{H}_{2}\rightarrow \mathbb{R}$ and consider a subgradient
variant of the $CQ$-method
\begin{equation}
x_{0}\in \mathcal{H}_{1};\quad x_{k+1}:=P_{c}\left( x_{k}+\frac{\lambda _{k}}{%
\Vert A\Vert ^{2}}A^{\ast }\Big(P_{q}(Ax_{k})-Ax_{k}\Big)\right) ,\quad
k=0,1,2,\ldots ,  \label{int:subCQmethod}
\end{equation}%
where both metric projections $P_{C}$ and $P_{Q}$ were formally replaced by
the corresponding subgradient projections $P_{c}$ and $P_{q}$; see Example %
\ref{e-subProj} for  a  precise definition. Yang \cite{Yan04}  shows that method \eqref{int:subCQmethod} converges  to some point in the solution set $C\cap A^{-1}(Q)\neq \emptyset $ in the finite dimensional setting and  when  $\lambda_{k}=\lambda \in (0,2)$. A weak convergence result in a Hilbert space was
later established by Xu in \cite{Xu10}. A simultaneous subgradient
projection algorithm for the MSSCFP  was also  considered in \cite{CMS07}.

Another, more general, situation may occur in the case of the \textit{split
common fixed point problem} \cite{CS09}, where
\begin{equation}
C:=\limfunc{Fix}S\quad \text{and}\quad Q:=\limfunc{Fix}T
\label{int:SCFP-sublevelb}
\end{equation}%
for given operators $S$ and $T$, which  are  assumed to be cutters; see
Definition \ref{def:QNE}. Again, by formally replacing  the  metric projections $%
P_{C}$ and $P_{Q}$ in \eqref{int:CQmethod} by $S$ and $T$, respectively, we
arrive at the following fixed-point variant of the $CQ$-method:
\begin{equation}
x_{0}\in \mathcal{H}_{1};\quad x_{k+1}:=S\left( x_{k}+\frac{\lambda _{k}}{%
\Vert A\Vert ^{2}}A^{\ast }\Big(T(Ax_{k})-Ax_{k}\Big)\right) ,\quad
k=0,1,2,\ldots .  \label{int:cutterCQmethod}
\end{equation}%
Obviously, since the subgradient projection is a cutter (see Example \ref%
{e-subProj}), method \eqref{int:cutterCQmethod} extends %
\eqref{int:subCQmethod}. Observe that, in general, \eqref{int:cutterCQmethod}
and \eqref{int:subCQmethod} are neither a variant of the PGM (no
differentiability), nor  of the  Krasnosel'ski{\u{\i}}-Mann method  because a cutter need not be nonexpansive. Nevertheless,  there are several results
showing the convergence of method \eqref{int:cutterCQmethod}; see, for
example, \cite{CS09} in the finite dimensional setting or \cite{Mou11}, \cite%
{WX11} and \cite{Ceg15} for weak convergence in Hilbert space. The main
assumption in the above-mentioned results is the demi-closedness of $%
\limfunc{Id}-S$ and $\limfunc{Id}-T$ at zero, which we call in this paper \textit{weak regularity} (see Definition \ref{def:Reg}). Note that only in
\cite[Example 6.1]{Ceg15} the relaxation parameters $\lambda _{k}\in
\lbrack \varepsilon ,2-\varepsilon ]$, whereas $\lambda _{k}=\lambda \in
(0,2) $ in \cite{CS09, Mou11, WX11}.

Observe that in all of the above $CQ$-type methods, the computation of the
next iterate requires knowing the operator norm $\|A\|$ or its estimation;
see, for example, \cite[Proposition 4.1]{Byr02}, where the norm of $A$ was
estimated for a sparse matrix $A$.  Other  solutions can be found, for
example, in \cite{QX05} and \cite{Yan05}. One of the simplest and the most
elegant one is due to L\'{o}pez et al. \cite{LMWX12}, who suggested to
consider the following variation of the $CQ$-method:
\begin{equation}  \label{int:ExtrPG}
x_0\in \mathcal{H}_1; \quad x_{k+1}:=P_C\left(x_k-\lambda_k\frac{2f(x_k)}{%
\|\nabla f(x_k)\|^2}\nabla f(x_k) \right), \quad k=0,1,2,\ldots,
\end{equation}
with $f$ defined in \eqref{int:f}. Observe that \eqref{int:ExtrPG} is
indeed a $CQ$-type method, which we call here the \textit{extrapolated $CQ$%
-method} since, after expanding, it can be explicitly written as
\begin{equation}  \label{int:ExtrCQ}
x_0\in \mathcal{H}_1; \quad x_{k+1}=P_C\left( x_k+\frac{\lambda_k \tau(x_k)}{%
\|A\|^2} A^\ast \Big(P_Q(Ax_k)-Ax_k\Big)\right), \quad k=0,1,2,\ldots,
\end{equation}
with $\tau(x)$ defined by
\begin{equation}  \label{int:simga}
\tau(x):=\frac{2\|A\|^2f(x)}{\|\nabla f(x)\|^2} = \left(\frac{\|A\|
\|P_Q(Ax)-Ax\|}{\|A^\ast(P_Q(Ax)-Ax)\|}\right)^2 \geq 1
\end{equation}
for $Ax\notin Q$ and $\tau(x):=1$ otherwise.  Observe that $x_{k+1}$ does not, in fact, depend on $\|A\|$.  Weak convergence of method %
\eqref{int:ExtrPG}  was  established in \cite{LMWX12} under the
assumption that $C\cap A^{-1}(Q)\neq\emptyset$ and $\lambda_k\in
[\varepsilon,2-\varepsilon]$. Again, by formally replacing $P_C$ and $P_Q$
in \eqref{int:ExtrCQ} and \eqref{int:simga} by weakly regular cutters $S$
and $T$, respectively, Cegielski  established  weak convergence of the above method
in \cite{Ceg16}.

Recently, Wang et al. \cite{WHLY17} have formulated a sufficient condition
for  a linear rate of convergence  of the extrapolated $CQ$-method \eqref{int:ExtrCQ}--%
\eqref{int:simga} in terms of bounded linear regularity of the SCFP, that
is,  when  for all $r>0$, there is $\gamma_r>0$  such that  for all $x\in C\cap B(0,r)$, we have
\begin{equation}  \label{int:BLRofSCFP}
\gamma_r d\big(x, C\cap A^{-1}(Q)\big) \leq d(Ax,Q).
\end{equation}
In particular, the above condition holds when $A(C)\cap \limfunc{int} Q
\neq\emptyset$; see \cite[Proposition 2.5]{WHLY17} for more details. We comment on this condition in connection with our work below; see Remark \ref%
{r-BLRofSCFP}.

\subsection*{Contribution of our paper}

Based on the above short overview, one could distinguish between three
different approaches  to the study of  the convergence properties of various $CQ$-type methods. The first one is viewed through the projected gradient method related to a certain objective $f$. The second one is a Krasnosel'ski{\u{\i}}%
-Mann approach with a certain averaged mapping. The last one is a more
general fixed point approach, where the metric projections $P_{C}$ and $%
P_{Q} $ are formally replaced by abstract operators $S$ and $T$. In this
paper we focus on the latter case.

To this end, we introduce the \textit{Landweber transform} $\mathcal{L}\{\cdot
\}$, which for a given operator $T\colon \mathcal{H}_{2}\rightarrow \mathcal{H%
}_{2}$ assigns an operator $\mathcal{L}\{T\}\colon \mathcal{H}%
_{1}\rightarrow \mathcal{H}_{1}$ defined by
\begin{equation}
\mathcal{L}\{T\}x:=x+\frac{1}{\Vert A\Vert ^{2}}A^{\ast }\big(T(Ax)-Ax\big), \quad
 x\in \mathcal H_1,
\label{int:Land}
\end{equation}%
which we call in this paper  the  \textit{Landweber operator}  corresponding  to $T$. Observe that we can rewrite \eqref{int:CQmethod}, \eqref{int:subCQmethod} and %
\eqref{int:cutterCQmethod} by using $\mathcal{L}\{P_{Q}\}$, $\mathcal{L}%
\{P_{q}\}$ and $\mathcal{L}\{T\}$, respectively. Moreover, one can show that
in all of the above-mentioned cases, we have
\begin{equation}
A^{-1}(Q)=\limfunc{Fix}\mathcal{L}\{P_{Q}\}=\limfunc{Fix}\mathcal{L}%
\{P_{q}\}=\limfunc{Fix}\mathcal{L}\{T\};
\end{equation}%
see \cite{Ceg15}. Thus the abstract study of  the  Landweber transform may indeed
contribute to the convergence analysis of various $CQ$-type methods.

The main purpose of this paper is to examine which properties of $T$ can be
preserved  by  the Landweber transform. In particular, it is known
that the Landweber transform of a (firmly/quasi) nonexpansive operator is
again (firmly/quasi) nonexpansive (see \cite[Lemma 3.1]{WX11}, \cite[Lemma 4.1]{Ceg15} and \cite[Proposition 4]{Ceg16}). Moreover, the transform of a weakly/boundedly regular operator preserves
the type of the regularity; see \cite[proof of Thm. 3.3]{WX11} and \cite[%
Lemma 4.1]{Ceg15} for the weak regularity, and \cite[Thm. 4.2]{CM16} for the
bounded regularity under  a compact operator  $A$ with closed range.

The main contribution of this paper is to formulate sufficient conditions
which  ensure that  the Landweber transform preserves bounded linear regularity which, as far as we know, is new. Moreover, we show that compactness of $A$ is no longer needed for the preservation of bounded regularity. For both of
these results see Theorem \ref{t-regLand}. In addition, based on %
\eqref{int:ExtrCQ} and \eqref{int:simga}, we consider the extrapolated
Landweber operator for which we establish similar results as in the
nonextrapolated case.

We would like to emphasize that by knowing the regularities of the operators
$S$ and $\mathcal{L}\{T\}$, in view of  the  recent paper \cite{CRZ18} (see Theorem \ref{t-regOper}), we  are able to  establish the corresponding weak, norm and linear convergence of $CQ$-type methods without
restricting ourselves just to projections. In particular, we formulate
sufficient conditions for linear convergence of  the  subgradient and cutter
methods described in \eqref{int:subCQmethod} and \eqref{int:cutterCQmethod}.
We comment on this in detail in Section \ref{s-SCFP}.

Finally, in order to formulate the linear rate more explicitly, we
investigate in detail the closed range theorem; see Lemma \ref{l-clR}.  To this end  we introduce  the  new quantity $|A|:=\inf \{\Vert Ax\Vert \mid x\in (\ker A)^{\perp },\ \Vert x\Vert =1\}$, which, for a matrix $A$  turns out to be  the square root of the smallest positive eigenvalue of the matrix $A^{\ast }A$.  In this connection, recall  that the spectral norm of the matrix $A$ is the square root of the
largest eigenvalue of $A^{\ast }A$. In addition, we show that, similarly to
the properties of $\Vert A\Vert $, we have $\left\vert A\right\vert
=\left\vert A^{\ast }\right\vert =\sqrt{\left\vert AA^{\ast }\right\vert }=%
\sqrt{\left\vert A^{\ast }A\right\vert }$.

\subsection*{Organization of our paper}

In Section \ref{s-preliminaries} we recall  several  necessary tools which  are  used in the establishing our main results. In Section \ref{s-ClRThm} we
present the closed range theorem. In Section \ref{s-LtO} we  formally introduce  the Landweber transform and investigate its properties. In Section %
\ref{s-ELtO} we adjust the results from the previous section to the
extrapolated Landweber operator. Finally, in Section \ref{s-SCFP}
we present a few convergence results for various $CQ$-type methods.

\section{Preliminaries}

\label{s-preliminaries}Let $\mathcal{H}$ be a real Hilbert space. We divide
the preliminaries into three separate subsections.

\subsection{Fej\'{e}r monotone sequences}

\begin{definition}
Let $F\subseteq\mathcal{H}$ be a nonempty, closed and convex set, and let $%
\{x_k\}_{k=0}^\infty$ be a sequence in $\mathcal{H}$. We say that $%
\{x_k\}_{k=0}^\infty$ is \textit{Fej\'er monotone} with respect to $F$ if
\begin{equation}
\|x_{k+1}-z\|\leq\|x_k-z\|
\end{equation}
for all $z\in F$ and every integer $k=0,1,2,\ldots$.
\end{definition}

\begin{theorem}
\label{t-Fejer} Let the sequence $\{x_{k}\}_{k=0}^{\infty } \subseteq \mathcal{H}$ be Fej\'{e}r monotone with respect to $F$. Then

\begin{enumerate}
\item[$(i)$] $\{x_k\}_{k=0}^\infty$ converges weakly to some point $%
x^\infty\in F$ if and only if all its weak cluster points lie in $F$.

\item[$(ii)$] $\{x_k\}_{k=0}^\infty$ converges strongly to some point $%
x^\infty\in F$ if and only if $d(x_k,F) \to 0$.

\item[$(iii)$] If there is some constant $q\in(0,1)$ such that $%
d(x_{k+1},F)\leq q d(x_k,F)$ holds for every $k=0,1,2,\ldots$, then $%
\|x_k-x^\infty\|\leq 2d(x_0,F)q^k$ for some $x_\infty \in F$.
\end{enumerate}
\end{theorem}

\begin{proof}
See, for example, \cite[Theorem 2.16 and Proposition 1.6]{BB96}.
\end{proof}

\subsection{Quasi-nonexpansive Operators}

\begin{definition}
\rm\ %
\label{def:QNE} Let $T:\mathcal{H}\rightarrow\mathcal{H}$ be an operator
with a fixed point, that is, $\limfunc{Fix} T \neq \emptyset$. We say that $%
T $ is

\begin{enumerate}
\item[$(i)$] \textit{quasi-nonexpansive} (QNE) if for all $x\in\mathcal{H}$
and all $z\in\limfunc{Fix} T$,
\begin{equation}
\| T x -z\|\leq\| x-z\|;  \label{eq:qne}
\end{equation}

\item[$(ii)$] $\rho$\textit{-strongly quasi-nonexpansive} ($\rho$-SQNE),
where $\rho\geq 0$, if for all $x\in\mathcal{H}$ and all $z\in\limfunc{Fix}
U $,
\begin{equation}
\| Tx-z\|^2\leq\| x-z\|^2-\rho\| Tx -x\|^2;  \label{eq:sqne}
\end{equation}

\item[$(iii)$] a \textit{cutter} if for all $x\in \mathcal{H}$ and all $z\in
\limfunc{Fix}T$,
\begin{equation}
\langle z-Tx,x-Tx\rangle \leq 0.  \label{eq:cutter}
\end{equation}
\end{enumerate}
\end{definition}

\begin{example}[Subgradient Projection]
\label{e-subProj}%
\rm\ %
\textrm{Let $f\colon \mathcal{H}\rightarrow \mathbb{R}$ be a lower semi-continuous and convex function with nonempty sublevel set $S(f,0):=\{x\in \mathcal{H}\mid f(x)\leq 0\}\neq
\emptyset $. For each $x\in \mathcal{H}$, let  $g_f(x)$  be a chosen subgradient
from the subdifferential set $\partial f(x):=\{g\in \mathcal{H}\mid f(y)\geq
f(x)+\langle g,y-x\rangle \text{, for all }y\in \mathcal{H}\}$,  which, by \cite[Proposition 16.27]{BC17}, is nonempty. We call the operator
$P_{f}\colon \mathcal{H}\rightarrow \mathcal{H}$ defined by
\begin{equation}
P_{f}(x):=%
\begin{cases}
\displaystyle x-\frac{f(x)}{\Vert g_{f}(x)\Vert ^{2}}g_{f}(x)\text{,} &
\text{if }f(x)>0, \\
x\text{,} & \text{otherwise.}%
\end{cases}%
\end{equation}%
a  \textit{subgradient projection}  related to $f$. It is not difficult to see
that $P_{f}$ is a cutter and $\limfunc{Fix}P_{f}=S(f,0)$ \cite[Corollary
4.2.6]{Ceg12}. }
\end{example}

\bigskip For a given \textit{relaxation function} $\alpha \colon \mathcal{H}%
\rightarrow (0,\infty )$ and an operator $T:\mathcal{H}\rightarrow \mathcal{H%
}$, we denote by $T_{\alpha }$  the  \textit{generalized $\alpha (\cdot )$%
-relaxation} of $T$ defined by
\begin{equation}
T_{\alpha }x:=x+\alpha (x)(Tx-x).  \label{e-genRel}
\end{equation}%
If $\alpha (x)=\alpha $, for some $\alpha >0$, then we simply call $%
T_{\alpha }$ an $\alpha $-\textit{relaxation}. In this paper we will
consider the former case in the context of extrapolation, where $\alpha
(x)\geq 1$; see Section \ref{s-ELtO}.

\begin{lemma}
\label{th:SQNEequiv} Let $T:\mathcal{H}\rightarrow \mathcal{H}$ be an
operator with $\limfunc{Fix}T\neq \emptyset $, let $\alpha \colon \mathcal{H}%
\rightarrow (0,\infty )$ and $\rho \geq 0$. The following conditions are
equivalent:

\begin{enumerate}
\item[$(i)$] $T$ is $\rho $-SQNE.

\item[$(ii)$] $T_{\frac{\rho +1}{2}}$ is a cutter.

\item[$(iii)$] For all $x\in \mathcal{H}$ and all $z\in \limfunc{Fix}T$,  we have
\begin{equation}
\langle Tx-x,z-x\rangle \geq \frac{\rho +1}{2}\Vert Tx-x\Vert ^{2}.
\label{eq:SQNEequiv:ineq1}
\end{equation}

\item[$(iv)$] For all $x\in \mathcal{H}$ and all $z\in \limfunc{Fix}T$,  we have
\begin{equation}
\Vert T_{\alpha }x-z\Vert ^{2}\leq \Vert x-z\Vert ^{2}-\left( \frac{\rho }{%
\alpha (x)}+\frac{1-\alpha (x)}{\alpha (x)}\right) \Vert T_{\alpha }x-x\Vert
^{2}.  \label{eq:SQNEequiv:ineq2}
\end{equation}
\end{enumerate}
\end{lemma}

\begin{proof}
See \cite[Corollary 2.1.43 and Remark 2.1.31]{Ceg12} for  the  equivalences
between (i), (ii) and (iii). The implication (i)$\Rightarrow $(iv) can be
found in \cite[Corollary 2.3 ]{BKRZ18}. Actually, slightly adjusting the
proof in \cite{BKRZ18}, one can deduce that the reverse implication is also
true.
\end{proof}

\begin{theorem}
\label{t-SQNE} Let $U_{i}\colon \mathcal{H}\rightarrow \mathcal{H}$ be $\rho
_{i}$-SQNE, $i=1,\ldots ,m$. Define the product operator $P:=\prod_{i=1}^{m}U_{i} :=U_{m}U_{m-1}...U_{1}$. Assume that $\rho :=\min_{i}\rho _{i}>0$ and $F:=\bigcap_{i=1}^{m}%
\limfunc{Fix}U_{i}\neq \emptyset $. Then $P$ is $(\rho /m)$-SQNE and $F=%
\limfunc{Fix}P$.
\end{theorem}

\begin{proof}
See \cite[Theorem 2.1.48]{Ceg12}.
\end{proof}

\subsection{Regular Families of Sets and Regular Operators}

The following definition can be found, for example, in \cite[Definition 5.1]%
{BB96} and \cite[Definition 5.7]{BNP15}.

\begin{definition}[Regular Sets]
\rm\ %
Let $S\subseteq \mathcal{H}$, $C_{i}\subseteq \mathcal{H}$, $i\in I:=\{1,\ldots ,m\}$, be closed  and  convex with $C:=\bigcap_{i\in I}C_{i}\neq \emptyset $ and let $\mathcal{C}%
:=\{C_{i}\mid i\in I\}$. We say that the family $\mathcal{C}$ is

\begin{enumerate}
\item[$\mathrm{(i)}$] \textit{regular} over $S$ if for any sequence $%
\{x_{k}\}_{k=0}^{\infty }\subseteq S$,  we have
\begin{equation*}
\lim_{k\rightarrow \infty }\max_{i\in I}d(x_{k},C_{i})=0\quad
\Longrightarrow \quad \lim_{k\rightarrow \infty }d(x_{k},C)=0\text{;}
\end{equation*}

\item[$\mathrm{(ii)}$] \textit{linearly regular} over $S$ if there is $%
\kappa _{S}>0$ such that for every $x\in S$,  we have
\begin{equation*}
d(x,C)\leq \kappa _{S}\max_{i\in I}d(x,C_{i})\text{.}
\end{equation*}%
The constant $\kappa _{S}$ is called a \textit{modulus} of the linear
regularity\textit{\ of }$\mathcal{C}$ over $S$.
\end{enumerate}

If any of the above regularity conditions holds for every subset $S\subseteq
\mathcal{H}$, then we simply omit the phrase \textquotedblleft over $S$". If
the same condition holds when restricted to bounded subsets $S\subseteq
\mathcal{H}$, then we precede the term with the adverb \textit{boundedly}.
\end{definition}

Below we list  a  few known examples of regular families of sets. For an extended list, see \cite{BB96} or \cite{BNP15}.

\begin{example}
\rm\ %
\textrm{\label{e-RegSets} Let $C_{i}\subseteq \mathcal{H}$, $i\in
I:=\{1,\ldots ,m\}$, be closed  and  convex with $C:=\bigcap_{i\in I}C_{i}\neq \emptyset $, and let $\mathcal{C}:=\{C_{i}\mid i\in I\}$. }

\begin{enumerate}
\item[$\mathrm{(i)}$] \textrm{If $\dim \mathcal{H}<\infty $, then $\mathcal{C%
}$ is boundedly regular; }

\item[$\mathrm{(ii)}$] \textrm{If all $C_{i}$, $i\in I$, are half-spaces,
then $\mathcal{C}$ is linearly regular; }

\item[$\mathrm{(iii)}$] \textrm{If $C_1\cap \limfunc{int}(\bigcap_{i=2}^m
C_{i})\neq \emptyset$, then $\mathcal{C}$ is boundedly linearly regular; }

\item[$\mathrm{(iv)}$] \textrm{If $\dim \mathcal{H}<\infty $, $C_{i}$ is a
half-space, $i=1,\ldots,p$, and $\bigcap_{i=1}^pC_{i}\cap \bigcap_{i=p+1}^m%
\limfunc{ri}C_{i}\neq \emptyset $, then $\mathcal{C}$ is boundedly linearly
regular. }
\end{enumerate}
\end{example}

\bigskip The following definition was introduced in \cite[Definitions 3.1
and 4.1]{CRZ18} (see also references therein for similar concepts).

\begin{definition}[Regular Operators]
\rm\ %
\label{def:Reg} Let $\{T_{k}\}_{k=0}^{\infty }$ be a sequence of operators $%
T_{k}\colon \mathcal{H}\rightarrow \mathcal{H}$ with $F:=\bigcap_{k=0}^{%
\infty }\limfunc{Fix}T_{k}\neq \emptyset $ and let $S\subseteq \mathcal{H}$
be nonempty. We say that $\{T_{k}\}_{k=0}^{\infty }$ is

\begin{enumerate}
\item[$\mathrm{(i)}$] \textit{weakly regular} over $S$ if for any sequence $%
\{x_{k}\}_{k=0}^{\infty }\subseteq S$ and for any  point  $x_{\infty }\in \mathcal{H}$,  we have
\begin{equation}
\left.
\begin{array}{l}
x_{n_{k}}\rightharpoonup x_{\infty } \\
T_{k}x_{k}-x_{k}\rightarrow 0%
\end{array}%
\right\} \quad \Longrightarrow \quad x_{\infty }\in F;  \label{eq:def:WReg}
\end{equation}

\item[$\mathrm{(ii)}$] \textit{regular} over $S$ if for any sequence $%
\{x_{k}\}_{k=0}^{\infty }\subseteq S$,  we have
\begin{equation}
\lim_{k\rightarrow \infty }\Vert T_{k}x_{k}-x_{k}\Vert =0\quad
\Longrightarrow \quad \lim_{k\rightarrow \infty }d(x_{k},F)=0;
\label{eq:def:Reg}
\end{equation}

\item[$\mathrm{(iii)}$] \textit{linearly regular} over $S$ if there is $%
\delta _{S}>0$ such that for every  point  $x\in S$,  we have
\begin{equation}
\Vert T_{k}x-x\Vert \geq \delta _{S}d(x,F).  \label{eq:def:LReg}
\end{equation}%
The constant $\delta _{S}$ is called a \textit{modulus} of the linear
regularity\textit{\ }of\textit{\ }$T$ over $S$.
\end{enumerate}

If any of the above regularity conditions holds for every subset $S\subseteq
\mathcal{H}$, then we simply omit the phrase \textquotedblleft over $S$". If
the same condition holds when restricted to bounded subsets $S\subseteq
\mathcal{H}$, then we precede the term with the adverb \textit{boundedly}.
Since there is no need to distinguish between boundedly weakly and weakly
regular operators, we call both weakly regular. If any of the above
regularities holds for a constant sequence with $T_{k}=T$ for some $T$, then
we simply refer to the regularity of the operator $T$ defined by the
corresponding regularity of the sequence $\{T\}_{k=0}^{\infty }$.
\end{definition}

Obviously the metric projection is linearly regular and thus (weakly) regular. The example
below shows that in some cases the subgradient projection  has similar properties.  The proof of part (i) can be found, for example, in
\cite[Theorem 4.2.7]{Ceg12}). The proof of part (iii) in the finite
dimensional setting can be found in \cite[Example 3.5 (iii)]{CRZ18}, which we
extend here to the infinite dimensional case.

\begin{example}[Regularity of Subgradient Projection]
\rm\ %
\label{e-subProjReg} \textrm{Let $P_{f}$ be a subgradient projection  as  defined in Example \ref{e-subProj}. Assume that $%
\partial f$ is uniformly bounded on bounded sets (see Remark \ref%
{r-subdifferential}). Then the following statements hold: }

\begin{enumerate}
\item[$\mathrm{(i)}$] \textrm{$P_{f}$ is weakly regular. }

\item[$\mathrm{(ii)}$] \textrm{If $f$ is $\alpha $-strongly convex, where $%
\alpha >0$, then $P_{f}$ is boundedly regular. }

\item[$\mathrm{(iii)}$] \textrm{If $f(z)<0$ for some $z$, then $P_{f}$ is
boundedly linearly regular. }
\end{enumerate}

\textrm{In particular, if $\mathcal{H}$ is finite dimensional, then $%
\partial f$ is uniformly bounded on bounded sets and by the equivalence
between weak and strong convergence, $P_f$ is boundedly regular. }
\end{example}

\begin{proof}
Let $B(z,r)$ be a ball with center $z\in S(f,0)$ and radius $r>0$. In order to show the weak/bounded/bounded linear regularity of $P_{f}$, it suffices to show the
corresponding regularity over $B(z,r)$ for arbitrary $r>0$.

Before proceeding, let us observe that for any $x\in S(f,0)$, we have
\begin{equation}
0=f_{+}(x)=\Vert P_{f}x-x\Vert =d(x,S(f,0)).  \label{pr-subProjReg-zero}
\end{equation}%
Moreover, by assumption, for any $r>0$, there is $M>0$ such that $\Vert g(x)\Vert \leq M$ for all $x\in B(z,r)$ and all $g(x)\in \partial f(x)$. Hence, for all $x\in B(z,r)$ such that $f(x)>0$, we have
\begin{equation}
\Vert P_{f}x-x\Vert \geq \frac{f(x)}{\Vert g(x)\Vert }\geq \frac{f_{+}(x)}{M}.
\label{pr-subProjReg-M}
\end{equation}

\textit{Part (i). } Let $r>0$ be arbitrary. Assume that $B(z,r)\ni x_{n_{k}}\rightharpoonup x_{\infty }$ and $\Vert
P_{f}x_{k}-x_{k}\Vert \rightarrow 0$. By combining \eqref{pr-subProjReg-zero}
and \eqref{pr-subProjReg-M}, we get $f_{+}(x_{k})\rightarrow 0$.  Since a lower semi-continuous and convex function is weakly lower semi-continuous \cite[Theorem 9.1]{BC17}, we have $0=\liminf_{k}f_{+}(x_{n_{k}}) \geq \liminf_{k}f(x_{n_{k}})\geq f(x_{\infty })$, that is, $x_{\infty }\in \limfunc{Fix}P_{f}$. This shows that $P_{f}$ is
weakly regular over $B(z,r)$.

\textit{Part (ii).} Let $r>0$ be arbitrary. Recall that $f$ is $\alpha $-strongly convex if
\begin{equation}
f(\lambda x+(1-\lambda )y)\leq \lambda f(x)+(1-\lambda )f(y)-\frac{\alpha }{2%
}\lambda (1-\lambda )\Vert x-y\Vert ^{2}
\end{equation}%
for all $x,y\in \mathcal{H}$ and $\lambda \in \lbrack 0,1]$. In particular,
for any $x\notin S(f,0)$, by setting $y=P_{S(f,0)}x$ and $\lambda =\frac{1}{2%
}$, we have
\begin{equation}
f_{+}(x)=f(x)\geq 2\underbrace{f\left( \frac{1}{2}x+\frac{1}{2}%
P_{S(f,0)}x\right) }_{>0}-\underbrace{f(P_{S(f,0)}x)}_{=0}+\frac{\alpha }{4}\Vert x-P_{S(f,0)}x\Vert ^{2}\geq \frac{\alpha }{4}d^{2}(x,S(f,0)).  \label{pr-subProjReg-strongConv}
\end{equation}%
This, when combined with \eqref{pr-subProjReg-zero} and %
\eqref{pr-subProjReg-M}, easily leads to  the  regularity of $P_{f}$ over $B(z,r)$.

\textit{Part (iii).} Assume that $f(z)<0$ and let $x\in B(z,r)$ be such that
$f(x)>0$. Define $y:=\lambda z+(1-\lambda )x$, where $\lambda :=\frac{f(x)}{%
f(x)-f(z)}>0$. Then $y\in S(f,0)$ and, by the definition of the metric
projection,  we have
\begin{equation}
d(x,S(f,0))\leq \Vert x-y\Vert =\lambda \Vert x-z\Vert \leq \frac{f_{+}(x)}{%
-f(z)}r.
\end{equation}%
The inequality above, when combined with \eqref{pr-subProjReg-zero} and %
\eqref{pr-subProjReg-M}, leads to linear regularity of $P_{f}$ over $B(z,r)$.
\end{proof}

\begin{remark}
\rm\ %
\label{r-subdifferential} Let $f\colon \mathcal{H}\rightarrow
\mathbb{R}$ be lower semi-continuous and convex. We recall that $\partial f$ is uniformly bounded on bounded sets if and only if $f$ is Lipschitz continuous on bounded sets,  a condition  which holds if and only if $f$ maps bounded sets onto bounded sets; see, for example, \cite[Proposition 7.8]{BB96}. The above holds true, in particular, when $\mathcal{H}=\mathbb{R}^{n}$.
\end{remark}

 The  theorem below summarizes some of the properties of regular operators discussed in \cite{CRZ18}. We apply this theorem in Section \ref{s-SCFP}
only for two sequences while studying the convergence of various $CQ$-type methods.

\begin{theorem}
\label{t-regOper} Let $\{U_{1}^{k}\}_{k=0}^{\infty },\ldots, \{U_{m}^{k}\}_{k=0}^{\infty }$ be given sequences of $\rho _{i}^{k}$-SQNE operators $U_{i}^{k}\colon \mathcal{H}\rightarrow \mathcal{H}$ with $\limfunc{Fix}U_{i}^{k}=F_{i}$, $i=1,2,...,m$, (in particular each sequence can be  a  constant consisting only of one operator $U_{i}$). Define the product operator $P_{k}$ by
\begin{equation}
P_{k}:=\prod_{i=1}^{m}U_{i}^{k}.
\end{equation}%
Assume that $\rho :=\inf_{i,k}\rho _{i}^{k}>0$ and $F:=\bigcap_{i=1}^{m}F_{i}\neq \emptyset $ ($P_{k}$ is $(\rho /m)$-SQNE and $F=\limfunc{Fix}P_{k}$ by Theorem \ref{t-SQNE}). Let $B=B(z,r)$ for some $z\in F$ and $r>0$. Then the following statements
hold:

\begin{enumerate}
\item[$\mathrm{(i)}$] If for each $i=1,\ldots,m$, the sequence $%
\{U_i^k\}_{k=0}^\infty$ is weakly regular over $B$, then $%
\{P_k\}_{k=0}^\infty$ is also weakly regular over $B$.

\item[$\mathrm{(ii)}$] If for each $i=1,\ldots,m$, the sequence $%
\{U_i^k\}_{k=0}^\infty$ is regular over $B$ and the family of sets $%
\{F_1,\ldots,F_m\}$ is regular over $B$, then $\{P_k\}_{k=0}^\infty$ is also
regular over $B$.

\item[$\mathrm{(iii)}$] If for each $i=1,\ldots,m$, the sequence $%
\{U_i^k\}_{k=0}^\infty$ is linearly regular over $B$ with modulus $\delta_i$ and the family of sets $\{F_1,\ldots,F_m\}$ is linearly regular over $B$
with modulus $\kappa>0$, then $\{P_k\}_{k=0}^\infty$ is also linearly
regular over $B$ with modulus
\begin{equation}
\delta_P=\frac{\rho \delta^2}{2m\kappa^2},
\end{equation}
where $\delta:=\min_i\delta_i.$
\end{enumerate}
\end{theorem}

\begin{proof}
 This  theorem is a particular case of \cite[Corollary 5.5]{CRZ18}.
\end{proof}

\begin{remark}[Erratum to \protect\cite{CRZ18}]
\rm\ %
\textrm{\ Note that  in contrast  to Definition \ref{def:Reg}(iii), in \cite%
{CRZ18} the constant $\delta _{S}^{-1}$ was called a modulus of the linear
regularity of $T$ over $S$. Consequently, in all the results of Section 5 in
\cite{CRZ18} one should use \textquotedblleft $\delta :=\max_{i}\delta _{i}$%
\textquotedblright\ instead of \textquotedblleft $\delta :=\min_{i}\delta
_{i}$\textquotedblright . The statement of Theorem \ref{t-regOper}  corrects this  unfortunate misprint. }
\end{remark}

\section{Closed Range Theorem}

\label{s-ClRThm}

Let $\mathcal{H}_{1}$ and $\mathcal{H}_{2}$ be two Hilbert spaces, $A:%
\mathcal{H}_{1}\rightarrow \mathcal{H}_{2}$ be a nonzero bounded linear
operator and $A^{\ast }:\mathcal{H}_{2}\rightarrow \mathcal{H}_{1}$ be its
adjoint operator defined by
\begin{equation*}
\langle Ax,y\rangle =\langle x,A^{\ast }y\rangle \qquad \text{for all }x\in
\mathcal{H}_{1}\text{ and all }y\in \mathcal{H}_{2}\text{.}
\end{equation*}%
We denote: $\ker A:=\{x\in \mathcal{H}_{1}\mid Ax=0\}$ -- the \textit{kernel}
(\textit{null space}) of $A$, $\limfunc{im}A:=\{y\in \mathcal{H}_{2}\mid
Ax=y $ for some $x\in \mathcal{H}_{1}\}$ -- the \textit{image} (\textit{range%
}) of $A$, $\limfunc{cl}C$ -- the \textit{closure} of a subset $C\subseteq
\mathcal{H}_{1}$, $V^{\bot }:=\{y\in \mathcal{H}_{1}\mid \langle x,y\rangle
=0$ for all $x\in V\}$ -- the \textit{orthogonal complement} of a subspace $%
V\subseteq \mathcal{H}_{1}$. It is not difficult to see that
\begin{equation}
\ker A=\ker A^{\ast }A\qquad \text{and}\qquad \ker A^{\ast }=\ker AA^{\ast }.
\label{e-kerAA*}
\end{equation}%
Moreover, it is well known that
\begin{equation}
(\ker A)^{\bot }=\limfunc{cl}(\limfunc{im}A^{\ast })\qquad \text{and}\qquad
(\ker A^{\ast })^{\bot }=\limfunc{cl}(\limfunc{im}A)\text{;}
\label{e-kerA_imA*}
\end{equation}%
see, for example, \cite[Lemma 8.33]{Deu01}. In what follows the following
property  turns out to  be useful.

\begin{proposition}
\label{l-bij}The operator $A^{\#}:(\ker A)^{\bot }\rightarrow \limfunc{im}A$
defined by $A^{\#}:=A\mid _{(\ker A)^{\bot }}$ is a bijection.
\end{proposition}

\begin{proof}
To see that $A^{\#}$ is surjective, let $y\in \limfunc{im}A$ and $x\in
\mathcal{H}_{1}$ be such that $y=Ax$. Then $x=x^{\prime }+x^{\prime \prime }$
with $x^{\prime }\in \ker A$ and $x^{\prime \prime }\in (\ker A)^{\bot }$.
We have $y=Ax=A(x^{\prime }+x^{\prime \prime })=Ax^{\prime \prime
}=A^{\#}x^{\prime \prime }$, which proves that $\limfunc{im}A^{\#}=\limfunc{%
im}A$. To see that $A^{\#}$ is injective, assume that $A^{\#}x_{1}=A^{%
\#}x_{2}$ for some $x_{1},x_{2}\in (\ker A)^{\bot }$. Then $%
A^{\#}(x_{1}-x_{2})=A(x_{1}-x_{2})=0$, that is, $x_{1}-x_{2}\in \ker A$.
Since we also have $x_{1}-x_{2}\in (\ker A)^{\bot }$ and $\ker A\cap (\ker
A)^{\bot }=\{0\}$, it follows that $x_{1}=x_{2}$. This proves that $A^{\#}$
is a bijection.
\end{proof}

\bigskip By $\Vert A\Vert :=\sup \{\Vert Ax\Vert \mid x\in \mathcal{H}%
_{1},\Vert x\Vert =1\}$ we denote the norm of the operator $A$. It is not
difficult to see that the norm of $A$ satisfies
\begin{equation}
\Vert A\Vert =\sup \{\Vert Ax\Vert \mid x\in (\ker A)^{\bot },\Vert x\Vert
=1\}\text{.}
\end{equation}%
Moreover,
\begin{equation}
\Vert A\Vert =\Vert A^{\ast }\Vert =\sqrt{\Vert A^{\ast }A\Vert }=\sqrt{%
\Vert AA^{\ast }\Vert }\text{;}  \label{e-normA*A}
\end{equation}%
see, for example, \cite[Theorem 8.25]{Deu01}. Analogously, we define
\begin{equation}
\left\vert A\right\vert :=\inf \{\Vert Ax\Vert \mid x\in (\ker A)^{\bot
},\Vert x\Vert =1\}\text{.}  \label{e-Ainf}
\end{equation}%
Clearly, for any $x\in (\ker A)^{\bot }$, we have
\begin{equation}
\left\vert A\right\vert \cdot \Vert x\Vert \leq \Vert Ax\Vert \leq \Vert
A\Vert \cdot \Vert x\Vert \text{,}  \label{e-Ainf2}
\end{equation}%
where $\left\vert A\right\vert $ and $\Vert A\Vert $ are the largest and the
smallest constants, respectively, for which the above inequalities hold. The
following lemma provides a list of equivalent conditions for $\left\vert
A\right\vert $ to be greater than zero. The equivalence (i)$\Leftrightarrow $%
(iv) is known as the \textit{closed range theorem } \cite[Chapter VII,
Section 5]{Yos78}.

\begin{lemma}[Closed Range Theorem]
\label{l-clR}Let $A:\mathcal{H}_{1}\rightarrow \mathcal{H}_{2}$ be a nonzero
bounded linear operator. The following statements are equivalent:
\begin{multicols}{3}
\begin{enumerate}
    \item[$\mathrm{(i)}$] $\limfunc{im}A$ is closed;
    \item[$\mathrm{(iv)}$] $\limfunc{im}A^{\ast }$ is closed;
    \item[$\mathrm{(vii)}$] $\limfunc{im}AA^{\ast }$ is closed;
    \item[$\mathrm{(x)}$] $\limfunc{im}A^{\ast }A$ is closed;

    \item[$\mathrm{(ii)}$] $\limfunc{im}A=(\ker A^{\ast })^{\bot }$;
    \item[$\mathrm{(v)}$] $\limfunc{im}A^{\ast }=(\ker A)^{\bot }$;
    \item[$\mathrm{(viii)}$] $\limfunc{im}AA^{\ast }=(\ker A A^{\ast })^{\bot }$;
    \item[$\mathrm{(xi)}$] $\limfunc{im}A^{\ast }A=(\ker A^{\ast } A)^{\bot }$;

    \item[$\mathrm{(iii)}$] $\left\vert A\right\vert >0$;
    \item[$\mathrm{(vi)}$] $\left\vert A^{\ast }\right\vert >0$;
    \item[$\mathrm{(ix)}$] $\left\vert A A^{\ast }\right\vert >0$;
    \item[$\mathrm{(xii)}$] $\left\vert A^{\ast }A\right\vert >0$.
\end{enumerate}
\end{multicols}\noindent Moreover, if any of the above conditions holds,
then (compare with \eqref{e-kerAA*})
\begin{equation}
\limfunc{im}AA^{\ast }=\limfunc{im}A\qquad \text{and}\qquad \limfunc{im}%
A^{\ast }A=\limfunc{im}A^{\ast }.  \label{e-imAA*}
\end{equation}%
Furthermore, in this case $A^{\#}$ has a bounded inverse and (compare with %
\eqref{e-normA*A})
\begin{equation}
\left\vert A\right\vert =\left\vert A^{\ast }\right\vert =\sqrt{\left\vert
AA^{\ast }\right\vert }=\sqrt{\left\vert A^{\ast }A\right\vert }=\frac{1}{%
\Vert (A^{\#})^{-1}\Vert }\text{.}  \label{e-AB}
\end{equation}
\end{lemma}

\begin{proof}
\textit{Step 1.} We begin the proof by showing that conditions (i)-(xii) are
indeed equivalent.

The equivalence (i)$\Leftrightarrow $(ii) follows from \eqref{e-kerA_imA*}
and the fact that the orthogonal complement of a nonempty set in a Hilbert
space is a closed subspace; see, for example, \cite[Theorem 4.5]{Deu01}. For
the equivalence (i)$\Leftrightarrow $(iii), see \cite[Theorem 8.18]{Deu01}.
Hence, we have obtained the equivalences (i)$\Leftrightarrow $(ii)$%
\Leftrightarrow $(iii). By replacing $A$ with $A^{\ast }$, $AA^{\ast }$ and $%
A^{\ast }A$ in the latter equivalence, we obtain (iv)$\Leftrightarrow $(v)$%
\Leftrightarrow $(vi), (vii)$\Leftrightarrow $(viii)$\Leftrightarrow $(ix)
and (x)$\Leftrightarrow $(xi)$\Leftrightarrow $(xii), respectively. For the
equivalence of statements (i) and (iv), see \cite[Chapter VII, Section 5]%
{Yos78}. Thus statements (i)-(vi) are equivalent. Next we show that (iv)$%
\Leftrightarrow $(vii). Indeed, assume that $\limfunc{im}A^{\ast }$ is
closed. Then,  the equality  $\limfunc{im}A^{\ast }=(\ker A)^{\bot }$ and the orthogonal
decomposition theorem yield
\begin{equation}
\limfunc{im}(AA^{\ast })=A(\limfunc{im}A^{\ast })=A((\ker A)^{\bot
})=A((\ker A)^{\bot }\oplus \ker A)=A(\mathcal{H}_{1})=\limfunc{im}A,
\end{equation}%
which is closed by the equivalence between (i) and (iv). This proves that
(iv)$\Rightarrow $(vii). Now assume that $\limfunc{im}AA^{\ast }$ is closed.
We prove that $\limfunc{im}A$ and, consequently, $\limfunc{im}A^{\ast }$ are
closed. Indeed, by \eqref{e-kerAA*} and \eqref{e-kerA_imA*}, we have
\begin{equation}
\limfunc{im}A\subseteq \limfunc{cl}(\limfunc{im}A)=(\ker A^{\ast })^{\perp
}=(\ker AA^{\ast })^{\perp }=\limfunc{im}AA^{\ast }\subseteq \limfunc{im}A.
\end{equation}%
Thus, $\limfunc{cl}(\limfunc{im}A)=\limfunc{im}A$. By interchanging $A$ with $A^{\ast }$ in the above argument, we can show that $\limfunc{im}A^{\ast }A=\limfunc{im}A^{\ast }$ and that (i)$ \Leftrightarrow $(x). Consequently, we have established the equivalence of
all the conditions (i)-(xii) and proved (\ref{e-imAA*}). In the  remaining  part of the proof we suppose that any of  the equivalent  conditions (i)-(xii) is satisfied.

\textit{Step 2.} Suppose that $\limfunc{im}A$ is closed. Then $\limfunc{im}A=(\ker A^{\ast })^{\bot }$. This means that $A^{\#}$ has
closed image, that is, $\limfunc{im}A^{\#}=\limfunc{im}A=(\ker A^{\ast
})^{\bot }$. This and the fact that $A^{\#}$ is a bijection yield that $%
A^{\#}$ has a bounded inverse $(A^{\#})^{-1}:(\ker A^{\ast })^{\bot
}\rightarrow (\ker A)^{\bot }$ (see \cite[Theorem 8.19]{Deu01}).

\textit{Step 3.} By assumption, $\limfunc{im}A^{\ast }$ is closed, that is, $%
\limfunc{im}A^{\ast }=(\ker A)^{\bot }$. Similarly to the definition of $%
A^{\#}$, we can define the operator $A^{\flat }:(\ker A^{\ast })^{\bot
}\rightarrow (\ker A)^{\bot }$ by $A^{\flat }:=A^{\ast }\mid _{(\ker A^{\ast
})^{\bot }}$. Moreover, similarly to Proposition \ref{l-bij} and Step 2 we
can prove that $A^{\flat }$ is a bijection and has a bounded inverse $%
(A^{\flat })^{-1}:(\ker A)^{\bot }\rightarrow (\ker A^{\ast })^{\bot }$.
Finally, for any $x\in (\ker A)^{\bot }$ and $y\in (\ker A^{\ast })^{\bot }$%
, we have
\begin{equation}
\langle A^{\#}x,y\rangle =\langle Ax,y\rangle =\langle x,A^{\ast }y\rangle
=\langle x,A^{\flat }y\rangle \text{,}
\end{equation}%
which means that
\begin{equation}
A^{\flat }=(A^{\#})^{\ast }\text{\qquad and\qquad }A^{\#}=(A^{\flat })^{\ast
}\text{.}  \label{e-Ab}
\end{equation}%
By \cite[Theorem 8.31]{Deu01},  we have
\begin{equation}
(A^{\flat })^{-1}=((A^{\#})^{\ast })^{-1}=((A^{\#})^{-1})^{\ast }\text{.}
\label{e-D-1}
\end{equation}

\textit{Step 4.} Now we show that \eqref{e-AB} indeed holds. For arbitrary
 norm-one  $y\in \limfunc{im}A$ and $x\in (\ker A)^{\bot }$ with $A^{\#}x=y$,
the first inequality in (\ref{e-Ainf2}) implies that
\begin{equation}
\Vert (A^{\#})^{-1}y\Vert =\Vert (A^{\#})^{-1}A^{\#}x\Vert =\Vert x\Vert
\leq \left\vert A\right\vert ^{-1}\Vert Ax\Vert =\left\vert A\right\vert
^{-1}\Vert y\Vert =\left\vert A\right\vert ^{-1}\text{.}  \label{e-A-1}
\end{equation}%
Thus
\begin{equation}
\Vert (A^{\#})^{-1}\Vert
=  \sup_{\substack{ y\in \limfunc{im}A \\ \Vert y\Vert =1}}  \Vert (A^{\#})^{-1}y\Vert
\leq \left\vert A\right\vert ^{-1}
\end{equation}%
and $\left\vert A\right\vert \leq \Vert (A^{\#})^{-1}\Vert ^{-1}$. We show
that $\left\vert A\right\vert \geq \Vert (A^{\#})^{-1}\Vert ^{-1}$. Suppose,
to the contrary, that $\Vert (A^{\#})^{-1}\Vert ^{-1}>\left\vert
A\right\vert $. By the definition of $\left\vert A\right\vert $, for any $k$
there is $x_{k}\in (\ker A)^{\bot }$ with $\Vert x_{k}\Vert =1$ such that%
\begin{equation}
\Vert A^{\#}x_{k}\Vert =\Vert Ax_{k}\Vert \leq \left\vert A\right\vert +%
\frac{1}{k}\text{.}  \label{e-Bxk}
\end{equation}%
We have
\begin{eqnarray}
\left\vert A\right\vert &<&\Vert (A^{\#})^{-1}\Vert ^{-1}=\Vert
(A^{\#})^{-1}\Vert ^{-1}\Vert x_{k}\Vert =\Vert (A^{\#})^{-1}\Vert
^{-1}\Vert (A^{\#})^{-1}A^{\#}x_{k}\Vert  \notag \\
&\leq &\Vert (A^{\#})^{-1}\Vert ^{-1}\Vert (A^{\#})^{-1}\Vert \cdot \Vert
A^{\#}x_{k}\Vert \leq \left\vert A\right\vert +\frac{1}{k},
\end{eqnarray}%
which, by letting $k\rightarrow \infty $, leads to a contradiction. Thus
\begin{equation}
\left\vert A\right\vert =\frac{1}{\Vert (A^{\#})^{-1}\Vert }.  \label{e-/A/}
\end{equation}%
Replacing $A$ with $A^{\ast }$,  we obtain,
\begin{equation}
\left\vert A^{\ast }\right\vert
=\frac{1}{\Vert (A^{\flat })^{-1}\Vert }.  \label{e-/A*/}
\end{equation}%
This, when combined with \eqref{e-D-1},  implies that
\begin{equation}
\left\vert A\right\vert
=\frac{1}{\Vert (A^{\#})^{-1}\Vert }
=\frac{1}{\Vert ((A^{\#})^{-1})^{\ast }\Vert }
 =\frac{1}{\Vert (A^{\flat })^{-1}\Vert }
=\left\vert A^{\ast }\right\vert ,  \label{e-a}
\end{equation}%
which proves a part of \eqref{e-AB}.

Similarly as in Proposition \ref{l-bij} and in Step 2, and since $(\ker
A)^{\perp }=(\ker A^{\ast }A)^{\perp }$, the operator $(A^{\ast
}A)^{\#}:(\ker A)^{\bot }\rightarrow (\ker A)^{\bot }$, defined by $(A^{\ast
}A)^{\#}:=A^{\ast }A\mid _{(\ker A)^{\bot }}$, is a bijection and has a
bounded inverse $((A^{\ast }A)^{\#})^{-1}$. By \cite[Theorem 8.25]{Deu01},  we have
\begin{equation}
\Vert (A^{\#})^{-1}((A^{\#})^{-1})^{\ast }\Vert =\Vert (A^{\#})^{-1}\Vert
^{2}\text{.}  \label{e-A+}
\end{equation}%
Moreover, for $x\in (\ker A)^{\bot }$, we obtain
\begin{equation}
(A^{\ast }A)^{\#}x=A^{\ast }Ax=A^{\ast }A^{\#}x=A^{\flat }A^{\#}x
\end{equation}%
which, when combined with (\ref{e-D-1}) and (\ref{e-A+}), implies  that
\begin{equation}
\Vert ((A^{\ast }A)^{\#})^{-1}\Vert =\Vert (A^{\flat }A^{\#})^{-1}\Vert
=\Vert (A^{\#})^{-1}(A^{\flat })^{-1}\Vert =\Vert
(A^{\#})^{-1}((A^{\#})^{-1})^{\ast }\Vert =\Vert (A^{\#})^{-1}\Vert ^{2}%
\text{.}
\end{equation}%
Thus, replacing $A$ with $A^{\ast }A$ in \eqref{e-/A/}, we arrive at
\begin{equation}
\left\vert A^{\ast }A\right\vert =\frac{1}{\Vert ((A^{\ast
}A)^{\#})^{-1}\Vert }=\frac{1}{\Vert (A^{\#})^{-1}\Vert ^{2}}=\left\vert
A\right\vert ^{2}\text{.}
\end{equation}%
In the same way, one can prove that
\begin{equation}
\left\vert AA^{\ast }\right\vert =\frac{1}{\Vert (A^{\flat })^{-1}\Vert ^{2}}%
=\left\vert A^{\ast }\right\vert ^{2}=\left\vert A\right\vert ^{2}.
\end{equation}%
This completes the proof of Lemma \ref{l-clR}.
\end{proof}

\begin{remark}
\rm\ %
\textrm{\ Observe that even if $\limfunc{im}A$ is not closed then, due to
the equivalences between (iii), (vi), (ix) and (xii), the equalities from %
\eqref{e-AB} will take the following form: $\left\vert A\right\vert
=\left\vert A^{\ast }\right\vert =\sqrt{\left\vert AA^{\ast }\right\vert }=%
\sqrt{\left\vert A^{\ast }A\right\vert }=0$. }
\end{remark}

\begin{lemma}[Compact Operators and Closed Range]
\label{l-compOp} Let $A:\mathcal{H}_{1}\rightarrow \mathcal{H}_{2}$ be a
nonzero bounded linear operator and assume that $A$ is compact. Let $\Lambda
^{+}(A^{\ast }A)$ be the set of all positive eigenvalues of $A^{\ast }A$. Then
\begin{equation}
\Vert A^{\ast }A\Vert =\sup \Lambda ^{+}(A^{\ast }A)\in \Lambda ^{+}(A^{\ast
}A).
\end{equation}%
Moreover, we have the following alternative:

\begin{enumerate}
\item[(i)] The set $\Lambda^+(A^*A)$ is finite (for example, when $A$ is a
nonzero $m\times n$ matrix). Then
\begin{equation}
|A^*A|=\min \Lambda^+(A^*A)>0
\end{equation}
and consequently, the set $\limfunc{im}A$ is closed.

\item[(ii)] The set $\Lambda^+(A^*A)$ is countably infinite. Then
\begin{equation}
|A^*A|=\inf \Lambda^+(A^*A)=0
\end{equation}
and consequently, the set $\limfunc{im}A$ is not closed.
\end{enumerate}
\end{lemma}

\begin{proof}
The proof follows from the basic properties of compact self-adjoint
operators and the spectral decomposition theorem applied to $A^*A$; see, for
example \cite[Section 15.3]{Kre14} or \cite[Section 4.8]{DM05}. 
We remark in passing that the spectral decomposition theorem is usually presented in the setting of a complex Hilbert space. Nevertheless, one can obtain an analogous result in a real Hilbert space by using, for instance, a complexification argument combined with the fact that all eigenvalues of a self-adjoint operator are real. Hence, when referring to results from \cite[Section 15.3]{Kre14} we actually adjust them to the setting of a real Hilbert space.

First observe that since $A$ is compact,  the operators  $A^{\ast }$ and consequently $A^{\ast }A$ are also compact; see \cite[Theorem 4.12]{Kre14}. Moreover, since $A^{\ast }A$ is self-adjoint and positive semi-definite, all of its
eigenvalues are nonnegative and, by \cite[Theorems 15.11]{Kre14},
we indeed have $\Vert A^{\ast }A\Vert =\sup \Lambda ^{+}(A^{\ast }A)\in
\Lambda ^{+}(A^{\ast }A)$. By the spectral decomposition theorem \cite[%
Theorem 15.12]{Kre14}, we know that $A^{\ast }A$ has  either a finite or a countably infinite  set of eigenvalues accumulating only at zero,  say $\Lambda^+(A^\ast A)=\{\lambda_k\mid k\in K\}$ with $K=\{1,\ldots,N\}$ in the former and $K=\mathbb N$ in the latter case. In both cases we  may assume that $\lambda_k>\lambda_{k+1}$.  Moreover, the
eigenspaces $E_{k}:=\ker (\lambda _{k}\limfunc{Id}-A^{\ast }A)$
corresponding to different eigenvalues are orthogonal and finite
dimensional. Finally, for every point $x\in \mathcal{H}_{1}$, we have
\begin{equation}
x=  \sum_{k\in K } P_{k}x+P_{\ker A^{\ast }A}x\qquad \text{and}\qquad
A^{\ast }A=  \sum_{k\in K }  \lambda _{k}P_{k},
\label{pr:compOp:decomposition}
\end{equation}%
where $P_{k}$ is the orthogonal projection  of $\mathcal H_1$  onto the eigenspace $E_{k}$.

\textit{Part (i).} Assume that  $\Lambda ^{+}(A^{\ast }A)$ is finite.  Observe that for any $x\in (\ker A^{\ast }A)^{\perp }$, we have $P_{\ker
A^{\ast }A}x=0$. Moreover, since the eigenspaces $E_{k}$ and $E_{l}$ are
orthogonal for $k\neq l$, the corresponding projections satisfy $P_{k}P_{l}=P_{l}P_{k}=0$.  Using  \eqref{pr:compOp:decomposition}, we arrive at
\begin{equation}
\Vert A^{\ast }Ax\Vert ^{2}
=\left\langle \sum_{k=1}^{N}\lambda _{k}P_{k}x,\ \sum_{l=1}^{N}\lambda_{l}P_{l}x \right\rangle
=\sum_{k=1}^{N}\lambda_{k}^{2}\Vert P_{k}x\Vert ^{2}
\geq  \lambda _{N}^{2}  \sum_{k=1}^{N}\Vert P_{k}x\Vert ^{2}
=  \lambda _{N}^{2}  \Vert x\Vert ^{2}
\end{equation}%
 for all $x\in \mathcal{H}_{1}$. Knowing that $|A^{\ast }A|$ is the largest number for which the above inequality holds (compare with \eqref{e-Ainf2}), we obtain $ \lambda _{N}  \leq |A^{\ast }A|$. On the other hand, for any norm-one eigenvector $  e_{N}\in E_{N} \subseteq (\ker A^{\ast }A)^{\perp }$, we have

\begin{equation}
\lambda _{N}=\Vert \lambda_{N}e_{N}\Vert =\Vert A^{\ast }Ae_{N}\Vert \geq
|A^{\ast }A|\cdot \Vert e_{N}\Vert =|A^{\ast }A|,
\end{equation}

which shows that $|A^{\ast }A|= \lambda _{N} $.

\textit{Part (ii).} Assume that  $\Lambda ^{+}(A^{\ast }A)$ is countably  infinite. For each $k=1,2,\ldots ,$ choose a
norm-one eigenvector $e_{k}\in E_{k}\subseteq (\ker A^{\ast }A)^{\perp }$.
 Then we have
\begin{equation}
|A^{\ast }A|=\inf \{\Vert A^{\ast }Ax\Vert \mid x\in (\ker A^{\ast }A)^{\bot
},\Vert x\Vert =1\}\leq \inf_{k}\Vert A^{\ast }Ae_{k}\Vert =\inf_{k}\lambda
_{k}=0,
\end{equation}%
where the last equality holds  because  zero is the only possible accumulation
point of $\Lambda ^{+}(A^{\ast }A)$. This completes the proof.
\end{proof}

\section{Landweber Operators\label{s-LtO}}

Let $\mathcal{H}_{1}$ and $\mathcal{H}_{2}$ be two Hilbert spaces, let $A:%
\mathcal{H}_{1}\rightarrow \mathcal{H}_{2}$ be a nonzero bounded linear
operator  and let  $T:\mathcal{H}_{2}\rightarrow \mathcal{H}_{2}$ be  an arbitrary  operator.

\begin{definition}
The operator $\mathcal{L}\{T\}:\mathcal{H}_{1}\rightarrow \mathcal{H}_{1}$
defined by
\begin{equation}
\mathcal{L}\{T\}x:=x+\frac{1}{\Vert A\Vert ^{2}}A^{\ast }\big(T(Ax)-Ax\big),
 \quad x\in \mathcal H_1,
\label{e-Land}
\end{equation}%
is called the \textit{Landweber operator} ( corresponding  to $T$). We call the operation $T\mapsto \mathcal{L}\{T\}$ the \textit{Landweber transform}.
\end{definition}

We recall that in the literature the Landweber operator is  usually  defined for $T=P_{Q}$, where $Q\subseteq \mathcal{H}_{2}$ is closed and convex (see,  for example,  \cite{Byr02}).

\begin{remark}
\rm\ %
Observe that the Landweber transform of the identity on $\mathcal{H}_{2}$ is
again the identity, but on $\mathcal{H}_{1}$. This can be written briefly as  follows:
\begin{equation}
\mathcal{L}\{\limfunc{Id}\}=\limfunc{Id}\text{.}
\end{equation}%
Moreover, for given operators $T:\mathcal{H}_{2}\rightarrow \mathcal{H}_{2}$
and $T_{i}:\mathcal{H}_{2}\rightarrow \mathcal{H}_{2}$,  weights $\omega _{i}\geq 0,i=1,2,...,m$, with $\sum_{i=1}^{m}\omega _{i}=1$ and a relaxation parameter
$\lambda \geq 0$,  the Landweber transform satisfies
\begin{equation}
\mathcal{L}\{T_{\lambda }\}=\mathcal{L}_{\lambda }\{T\}\text{,}
\label{e-LTrel}
\end{equation}%
where $T_{\lambda }:=\limfunc{Id}+\lambda (T-\limfunc{Id})$ and $\mathcal{L}%
_{\lambda }\{T\}:=\limfunc{Id}+\lambda (\mathcal{L}\{T\}-\limfunc{Id})$
denote the $\lambda $-relaxation of $T$ and $\mathcal{L}\{T\}$,
respectively, and
\begin{equation}
\mathcal{L}\left\{ \sum_{i=1}^{m}\omega _{i}T_{i}\right\}
=\sum_{i=1}^{m}\omega _{i}\mathcal{L}\{T_{i}\}.  \label{e-LTconv}
\end{equation}%
Furthermore, if $A$ is unitary (that is, when $A^{\ast }A=\limfunc{Id}$),
then $\mathcal{L}\{T\}=A^{\ast }TA$ and thus
\begin{equation}
\mathcal{L}\left\{ \prod_{i=1}^{m}T_{i}\right\} =\prod_{i=1}^{m}\mathcal{L}%
\{T_{i}\}.  \label{e-LTprod}
\end{equation}
\end{remark}

\bigskip
In order to formulate our next lemma,  we recall \cite{BBR78}  that an operator $T:\mathcal{H}\rightarrow \mathcal{H}$ is called $\alpha $-\textit{%
averaged} ($\alpha $-AV), where $\alpha \in (0,1)$, if $T$ is  the  $\alpha $-relaxation of some nonexpansive operator $U$, that is, $T=(1-\alpha )\limfunc{Id}+\alpha U$.

\begin{lemma}
\label{l-FixLT} Let $\mathcal{L}\{T\}$ be the Landweber operator  corresponding  to $%
T:\mathcal{H}_{2}\rightarrow \mathcal{H}_{2}$, $\alpha \in (0,1)$ and $\rho
\geq 0$.

\begin{enumerate}
\item[$\mathrm{(i)}$] If $T$ (firmly) nonexpansive, then $\mathcal{L}\{T\}$
is also (firmly) nonexpansive.

\item[$\mathrm{(ii)}$] If $T$ is $\alpha $-AV, then $\mathcal{L}\{T\}$ is
also $\alpha $-AV.

\item[$\mathrm{(iii)}$] If $T$ is a cutter and $\limfunc{im}A\cap \limfunc{%
Fix}T\neq \emptyset $, then $\mathcal{L}\{T\}$ is also a cutter and $%
\limfunc{Fix}\mathcal{L}\{T\}=A^{-1}(\limfunc{Fix}T)$.

\item[$\mathrm{(iv)}$] If $T$ is $\rho $-SQNE and $\limfunc{im}A\cap
\limfunc{Fix}T\neq \emptyset $, then $\mathcal{L}\{T\}$ is also $\rho $-SQNE
and, as in (iii), we have $\limfunc{Fix}\mathcal{L}\{T\}=A^{-1}(%
\limfunc{Fix}T)$.
\end{enumerate}
\end{lemma}

\begin{proof}
(i) If $T$ is firmly nonexpansive, then the result follows  from  \cite[Proposition 4]{Ceg16}. Now suppose that $T$ is nonexpansive,  that is,  $T=\limfunc{Id}+2(U-\limfunc{Id})$ for some firmly nonexpansive operator $U$; see, for example, \cite[Theorem 2.2.10]{Ceg12}. Hence, by \eqref{e-LTrel},
$\mathcal{L}\{T\}=\mathcal{L}\{U_{2}\}=\mathcal{L}_{2}\{U\}$, which is
nonexpansive. (ii) By \cite[Corollary 2.2.17]{Ceg12}, $T$ is $\alpha $-AV if
and only if $T=\limfunc{Id}+2\alpha (U-\limfunc{Id})$ for some firmly
nonexpansive operator $U$. Hence the  result  follows  from  (i)
and \eqref{e-LTrel}. A proof of (iii) can be found in \cite[Lemma 3.1]{WX11}. (iv) It follows  from  Lemma \ref{th:SQNEequiv} that $T$ is $\rho $-SQNE if and only if $T=\limfunc{Id}+\frac{2}{\rho +1}(U-\limfunc{Id})$ for some cutter $U$, which in this case equals $T_{(\rho +1)/2}$. Hence we can refer,  once again,  to \eqref{e-LTrel}
which, when combined with (iii), completes the proof. An independent proof
of parts (iii) and (iv) can also be found in \cite[Lemma 4.1]{Ceg15}.
\end{proof}

\bigskip

Observe that the equality $\limfunc{Fix}\mathcal{L}\{T\}=A^{-1}(\limfunc{Fix}%
T)$ yields the following equivalence:
\begin{equation}
A^{\ast }(T(Ax)-Ax)=0\Longleftrightarrow Ax\in \limfunc{Fix}T\text{.}
\label{e-FixT}
\end{equation}%
We use this fact later in Section \ref{s-ELtO} while defining the
extrapolated Landweber operator. Before formulating the main result of this
section, we prove  several  auxiliary  lemmata.

\begin{lemma}
\label{l-A1}Let $A:\mathcal{H}_{1}\rightarrow \mathcal{H}_{2}$ be a nonzero
bounded linear operator with closed $\limfunc{im}A$ and let $T:\mathcal{H}%
_{2}\rightarrow \mathcal{H}_{2}$ be quasi-nonexpansive. Assume that $%
\limfunc{im}A\cap \limfunc{Fix}T\neq \emptyset $. Then for any $x\in
\mathcal{H}_{1}$, we have
\begin{equation}
\frac{1}{\Vert A\Vert }d\big(Ax,\limfunc{im}A\cap \limfunc{Fix}T\big)\leq d%
\big(x,\limfunc{Fix}\mathcal{L}\{T\}\big)\leq \frac{1}{|A|}d\big(Ax,\limfunc{%
im}A\cap \limfunc{Fix}T\big)\text{.}  \label{e-dxFixV}
\end{equation}
\end{lemma}

\begin{proof}
To shorten our notation, denote $F:=\limfunc{im}A\cap \limfunc{Fix}T$. Let $%
x\in \mathcal{H}_{1}$. Clearly, $F$ is closed and convex and $d(Ax,F)=\Vert
Ax-P_{F}(Ax)\Vert $. Let $z:=P_{F}(Ax)$. By the equality $\limfunc{Fix}%
\mathcal{L}\{T\}=A^{-1}(\limfunc{Fix}T)$ (see Lemma \ref{l-FixLT}), there is
a point $w\in \limfunc{Fix}\mathcal{L}\{T\}$ such that $Aw=z$. Let
\begin{equation}
u:=x-w=u^{\prime }+u^{\prime \prime }\text{,}  \label{e-un}
\end{equation}%
where,  by the orthogonal decomposition theorem,  $u^{\prime }\in \ker A$ and $u^{\prime \prime }\in  (\ker A)^\perp $. Note that
\begin{equation}
w^{\prime }:=w+u^{\prime }\in \limfunc{Fix}\mathcal{L}\{T\}  \label{e-zn'}
\end{equation}%
because $Aw^{\prime }=Aw+Au^{\prime }=Aw=z\in \limfunc{Fix}T$. Moreover,
\begin{equation}
x-w^{\prime }=u^{\prime \prime }\in (\ker A)^{\bot }.
\end{equation}%
Thus (\ref{e-Ainf2}) yields%
\begin{equation}
\Vert Ax-Aw\Vert =\Vert Ax-Aw^{\prime }\Vert =\Vert A(x-w^{\prime })\Vert
=\Vert Au^{\prime \prime }\Vert \geq \left\vert A\right\vert \cdot \Vert
u^{\prime \prime }\Vert =\left\vert A\right\vert \cdot \Vert x-w^{\prime
}\Vert \text{.}  \label{eAx-Azn}
\end{equation}%
This, when combined with (\ref{e-zn'}) and the definition of the metric
projection, yields

\begin{equation}
d(Ax,F)=\Vert Ax-z\Vert =\Vert Ax-Aw\Vert \geq \left\vert A\right\vert \cdot
\Vert x-w^{\prime }\Vert \geq \left\vert A\right\vert d(x,\limfunc{Fix}%
\mathcal{L}\{T\}),
\end{equation}
which together with the equivalence (i)$\Leftrightarrow $(iii) in Lemma \ref%
{l-clR} proves the second inequality in \eqref{e-dxFixV}.

On the other hand, let $z=P_{\limfunc{Fix}\mathcal{L}\{T\}}x$. Since $Az\in
F $, we have
\begin{equation}
d(x,\limfunc{Fix}{\mathcal{L}\{T\}})=\Vert x-z\Vert \geq \frac{1}{\Vert
A\Vert }\Vert Ax-Az\Vert \geq \frac{1}{\Vert A\Vert }d(Ax,F)
\end{equation}%
and the proof is complete.
\end{proof}

\begin{corollary}
\label{l-A1forQ} Let $A:\mathcal{H}_{1}\rightarrow \mathcal{H}_{2}$ be a
nonzero bounded linear operator with closed $\limfunc{im}A$ and let $Q\subseteq \mathcal{H}_{2}$ be closed and convex. Assume that $%
\limfunc{im}A\cap Q\neq \emptyset $. Then for any $x\in \mathcal{H}_{1}$, we
have
\begin{equation}
\frac{1}{\Vert A\Vert }d\big(Ax,\limfunc{im}A\cap Q\big)=d\big(x,A^{-1}(Q)%
\big)\leq \frac{1}{|A|}d\big(Ax,\limfunc{im}A\cap Q\big)\text{.}
\label{e-dxFixV2}
\end{equation}
\end{corollary}

\begin{proof}
The result follows easily from Lemma \ref{l-A1} with $T=P_Q$.
\end{proof}

\begin{lemma}
\label{l-A2}Let $A:\mathcal{H}_{1}\rightarrow \mathcal{H}_{2}$ be a nonzero
bounded linear operator and let $T:\mathcal{H}_{2}\rightarrow \mathcal{H}%
_{2} $ be $\rho $-SQNE, where $\rho \geq 0$. Assume that $\limfunc{im}A\cap
\limfunc{Fix}T\neq \emptyset $. Then for any $x\in \mathcal{H}_{1}$, we have
\begin{equation}
\Vert T(Ax)-Ax\Vert ^{2}\leq \frac{2\Vert A\Vert ^{2}}{\rho +1}\Vert
\mathcal{L}\{T\}x-x\Vert \cdot d(x,\limfunc{Fix}\mathcal{L}\{T\})\text{.}
\label{e-TAx}
\end{equation}
\end{lemma}

\begin{proof}
Let $x\in \mathcal{H}_{1}$ and $z\in \limfunc{Fix}\mathcal{L}\{T\}$. We
recall that, by Lemma \ref{l-FixLT}, we have $\limfunc{Fix}\mathcal{L}%
\{T\}=A^{-1}(\limfunc{Fix}T)$. Hence $Az\in \limfunc{Fix}T$. By Lemma \ref%
{th:SQNEequiv} and by the Cauchy-Schwarz inequality, we have
\begin{eqnarray}
\Vert T(Ax)-Ax\Vert ^{2} &\leq &\frac{2}{\rho +1}\langle T(Ax)-Ax,\
Az-Ax\rangle   \notag \\
&=&\frac{2\Vert A\Vert ^{2}}{\rho +1}\left\langle \frac{1}{\Vert A\Vert ^{2}}%
A^{\ast }(T(Ax)-Ax),\ z-x\right\rangle   \notag \\
&\leq &\frac{2\Vert A\Vert ^{2}}{\rho +1}\Vert \mathcal{L}\{T\}x-x\Vert
\cdot \Vert z-x\Vert .
\end{eqnarray}%
Observe that for $z=P_{\limfunc{Fix}\mathcal{L}\{T\}}x$, we have $\Vert
z-x\Vert =d(x,\limfunc{Fix}\mathcal{L}\{T\})$, which completes the proof.
\end{proof}

\begin{theorem}
\label{t-regLand} Let $A:\mathcal{H}_{1}\rightarrow \mathcal{H}_{2}$ be a
nonzero bounded linear operator, let $T:\mathcal{H}_{2}\rightarrow \mathcal{H%
}_{2}$ be $\rho $-SQNE, where $\rho \geq 0$, and let $S\subseteq \mathcal{H}%
_{2}$ be nonempty. Assume that $\limfunc{im}A\cap \limfunc{Fix}T\neq
\emptyset $. Let $\mathcal{L}\{T\}:\mathcal{H}_{1}\rightarrow \mathcal{H}%
_{1} $ be the Landweber operator defined by \eqref{e-Land}. Then the
following statements hold:

\begin{enumerate}
\item[$\mathrm{(i)}$] If $T$ is weakly regular over $S$, then $\mathcal{L}%
\{T\}$ is weakly regular over $A^{-1}(S)$.

\item[$\mathrm{(ii)}$] If $\limfunc{im}A$ is closed, $T$ is regular over $S$%
, $S$ is bounded and the family $\{\limfunc{im}A,\limfunc{Fix}T\}$ is regular over $S$, then $\mathcal{L}%
\{T\}$ is regular over $A^{-1}(S)$.

\item[$\mathrm{(iii)}$] If $\limfunc{im}A$ is closed, $T$ is linearly
regular with modulus $\delta >0$ over $S$ and the family $\{\limfunc{im}A,%
\limfunc{Fix}T\}$ is linearly regular with modulus $\kappa >0$ over $S$,
then $\mathcal{L}\{T\}$ is linearly regular over $A^{-1}(S)$ with modulus
\begin{equation}
\Delta :=\frac{\rho +1}{2}\left( \frac{\delta |A|}{\kappa \Vert A\Vert }%
\right) ^{2}\text{,}  \label{e-Delta}
\end{equation}

that is, for any $x\in A^{-1}(S)$, we have
\begin{equation}
\Vert \mathcal{L}\{T\}x-x\Vert \geq \frac{\rho +1}{2}\left( \frac{\delta |A|%
}{\kappa \Vert A\Vert }\right) ^{2}d(x,\limfunc{Fix}\mathcal{L}\{T\})\text{.}
\label{e-DeltaInequality1}
\end{equation}
\end{enumerate}
\end{theorem}

\begin{proof}
For a bounded  sequence $ \{x_{k}\}_{k=0}^\infty\subseteq A^{-1}(S)$ and a point  $z\in \limfunc{Fix}\mathcal{%
L}\{T\}$, we have

\begin{equation}
\infty >R:=\sup_{k}\Vert x_{k}-z\Vert \geq d(x_{k},\limfunc{Fix}\mathcal{L}%
\{T\}).  \label{pr-RL(i)R}
\end{equation}%
By \eqref{e-TAx} and \eqref{pr-RL(i)R}, we obtain
\begin{eqnarray}
\Vert T(Ax_{k})-Ax_{k}\Vert ^{2} &\leq &\frac{2\Vert A\Vert ^{2}}{\rho +1}%
\Vert \mathcal{L}\{T\}x_{k}-x_{k}\Vert \cdot d(x_{k},\limfunc{Fix}\mathcal{L}%
\{T\})  \notag \\
&\leq &\frac{2R\Vert A\Vert ^{2}}{\rho +1}\Vert \mathcal{L}%
\{T\}x_{k}-x_{k}\Vert \text{.}  \label{pr-RL(i)ineq}
\end{eqnarray}

\textit{Part (i).} Let $\{x_{k}\}_{k=0}^\infty \subseteq A^{-1}(S)$ and $x\in \mathcal{H}%
_{1}$ be such that
\begin{equation}
x_{k}\rightharpoonup _{k}x\qquad \text{and}\qquad \Vert \mathcal{L}%
\{T\}x_{k}-x_{k}\Vert \rightarrow _{k}0,  \label{pr-RL(i)assumption}
\end{equation}%
and let $z\in \limfunc{Fix}\mathcal{L}\{T\}$. Since $\{x_{k}\}_{k=0}^\infty$  is bounded
as a weakly convergent sequence, (\ref{pr-RL(i)ineq}) and (\ref%
{pr-RL(i)assumption}) yield $\Vert T(Ax_{k})-Ax_{k}\Vert \rightarrow _{k}0$.
Observe that for any $y\in \mathcal{H}_{2}$, we get
\begin{equation}
\langle Ax_{k}-Ax,y\rangle =\langle x_{k}-x,A^{\ast }y\rangle \rightarrow
_{k}0\text{,}
\end{equation}%
and, consequently, $Ax_{k}\rightharpoonup Ax$. Since $Ax_{k}\in S$, by the
weak regularity of $T$, we obtain $Ax\in \limfunc{Fix}T$. The latter
statement is equivalent to $x\in A^{-1}(\limfunc{Fix}T)=\limfunc{Fix}%
\mathcal{L}\{T\}$, which completes the proof of part (i).

\textit{Part (ii).} Let $S$ be bounded and  let  $ \{x_{k}\}_{k=0}^\infty  \subseteq A^{-1}(S)$ be
such that
\begin{equation}
\Vert \mathcal{L}\{T\}x_{k}-x_{k}\Vert \rightarrow _{k}0\text{.}
\label{pr-RL(ii)assumption}
\end{equation}
By the second inequality in (\ref{e-dxFixV}), the sequence  $\{x_{k}\}_{k=0}^\infty$  is
bounded.  Using  (\ref{pr-RL(i)ineq}) and (\ref{pr-RL(ii)assumption}), we can
conclude that
\begin{equation}
\Vert TAx_{k}-Ax_{k}\Vert \rightarrow _{k}0\text{.}
\end{equation}%
Since $Ax_{k}\in S$,  using  the regularity of the operator $T$ over $S$, we
obtain
\begin{equation}
\max \big\{d(Ax_{k},\limfunc{Fix}T),\ d(Ax_{k},\limfunc{im}A)\big\}=d(Ax_{k},%
\limfunc{Fix}T)\rightarrow _{k}0\text{.}  \label{pr-RL(ii)limit}
\end{equation}%
By \eqref{e-dxFixV}, \eqref{pr-RL(ii)limit} and by the regularity of the
family $\{\limfunc{im}A,\limfunc{Fix}T\}$ over $S$, we arrive at
\begin{equation}
d(x_{k},\limfunc{Fix}\mathcal{L}\{T\})\leq \frac{1}{|A|}d(Ax_{k},\limfunc{im}%
A\cap \limfunc{Fix}T)\rightarrow _{k}0\text{.}  \label{pr-RL(ii)limit2}
\end{equation}%
Note that $|A|>0$ is guaranteed by the assumption that $\limfunc{im}A$ is
closed and by Lemma \ref{l-clR}. It is clear that \eqref{pr-RL(ii)limit2}
completes the proof of part (ii).

\textit{Part (iii).} Let $x\in A^{-1}(S)$ so that $Ax\in S$. By the linear
regularity of $T$ over $S$, the linear regularity of the family $\{\limfunc{%
im}A,\limfunc{Fix}T\}$ over $S$ and \eqref{e-dxFixV}, we get
\begin{eqnarray}
\Vert T(Ax)-Ax\Vert &\geq &\delta d(Ax,\limfunc{Fix}T)  \notag
\label{pr-RL(iii)ineq1} \\
&=&\delta \max \{d(Ax,\limfunc{Fix}T),\ d(Ax,\limfunc{im}A)\}  \notag \\
&\geq &\frac{\delta }{\kappa }d(Ax,\limfunc{im}A\cap \limfunc{Fix}T)  \notag
\\
&\geq &\frac{\delta }{\kappa }|A|d(x,\limfunc{Fix}\mathcal{L}\{T\})\text{.}
\end{eqnarray}%
Moreover, by \eqref{e-TAx},
\begin{equation}
\Vert T(Ax)-Ax\Vert ^{2}\leq \frac{2\Vert A\Vert ^{2}}{\rho +1}\Vert
\mathcal{L}\{T\}x-x\Vert \cdot d(x,\limfunc{Fix}\mathcal{L}\{T\}),
\label{pr-RL(iii)ineq2}
\end{equation}%
which, when combined with \eqref{pr-RL(iii)ineq1}, leads to (\ref%
{e-DeltaInequality1}). This completes the proof.
\end{proof}

\begin{remark}
\rm\ %
\label{r-LandRelCutter} Assume that $T$ is a cutter ($\rho =1$) and that all
the assumptions of Theorem \ref{t-regLand}(iii) are satisfied. Let $\lambda
\in (0,2]$. Then, for the relaxation $T_{\lambda }$ of $T$, inequality (\ref{e-DeltaInequality1}) takes the following form:
\begin{equation}
\Vert \mathcal{L}\{T_{\lambda }\}x-x\Vert =\lambda \Vert \mathcal{L}%
\{T\}x-x\Vert \geq \lambda \left( \frac{\delta |A|}{\kappa \Vert A\Vert }%
\right) ^{2}d(x,\limfunc{Fix}\mathcal{L}\{T_{\lambda }\})\text{,}
\label{e-DeltaInequalityCutter}
\end{equation}%
where $x\in S$. In particular, if $T=P_{Q}$, where $Q\subseteq \mathcal{H}%
_{2}$ is nonempty, closed and convex, and $Q\cap \limfunc{im}A\neq \emptyset
$, then
\begin{equation}
\Vert \mathcal{L}_{\lambda }\{P_{Q}\}x-x\Vert \geq \lambda \left( \frac{|A|}{%
\kappa \Vert A\Vert }\right) ^{2}d(x,A^{-1}(Q))\text{.}
\label{e-DeltaInequalityQ}
\end{equation}
\end{remark}

\begin{corollary}
\label{c-BoundRegLand} Let $A:\mathcal{H}_{1}\rightarrow \mathcal{H}_{2}$ be
a nonzero bounded linear operator and let $T:\mathcal{H}_{2}\rightarrow
\mathcal{H}_{2}$ be $\rho $-SQNE, where $\rho \geq 0$. Assume that $\limfunc{%
im}A\cap \limfunc{Fix}T\neq \emptyset $. Let $\mathcal{L}\{T\}:\mathcal{H}%
_{1}\rightarrow \mathcal{H}_{1}$ be the Landweber operator defined by %
\eqref{e-Land}. Then the following statements hold:

\begin{enumerate}
\item[$\mathrm{(i)}$] If $T$ is weakly regular, then $\mathcal{L }\{T\}$ is
weakly regular.

\item[$\mathrm{(ii)}$] If $\limfunc{im}A$ is closed, $T$ is boundedly
regular and $\{\limfunc{im}A,\limfunc{Fix}T\}$ is boundedly regular, then $%
\mathcal{L }\{T\}$ is boundedly regular.

\item[$\mathrm{(iii)}$] If $\limfunc{im}A$ is closed, $T$ is boundedly
linearly regular and the family $\{\limfunc{im}A,\limfunc{Fix}T\}$ is
boundedly linearly regular, then $\mathcal{L}\{T\}$ is boundedly linearly
regular.
\end{enumerate}
\end{corollary}

\begin{remark}
\label{r-regOfLand}
\rm\ %
Part (i) of Corollary \ref{c-BoundRegLand} was proved in \cite[Lemma 4.1]%
{Ceg15}. This part can also be deduced from the proof of \cite[Theorem 3.3]%
{WX11}. Part (ii) of Corollary \ref{c-BoundRegLand} was proved in \cite[%
Theorem 4.2]{CM16}, under the assumption that $A$ is compact and $\underline{%
\lambda }:=\inf \Lambda ^{+}(A^{\ast }A)>0$. Note that the latter result
follows from Corollary \ref{c-BoundRegLand}(ii). Indeed, by the alternative
(i) presented in Lemma \ref{l-compOp}, we see that $\underline{\lambda }>0$
only when $\limfunc{im}A$ is closed.
\end{remark}

\section{Extrapolated Landweber Operator\label{s-ELtO}}

Let $T:\mathcal{H}_{2}\rightarrow \mathcal{H}_{2}$ be a given operator and $%
\sigma :\mathcal{H}_{1}\rightarrow \lbrack 1,\infty )$ be an \textit{%
extrapolation function}.

\begin{definition}
\rm\ %
The operator $\mathcal{L}_{\sigma }\{T\}:\mathcal{H}_{1}\rightarrow \mathcal{%
H}_{1}$, defined by
\begin{equation}
\mathcal{L}_{\sigma }\{T\}x:=x+\sigma (x)\big(\mathcal{L}\{T\}x-x\big),
 \quad x\in \mathcal H_1,
\label{e-extLand0}
\end{equation}%
is called an \textit{extrapolated Landweber operator} ( corresponding  to $T$ and $%
\sigma $).
\end{definition}

In this section, following \cite{LMWX12} and \cite{CM16} (see the Introduction),
we consider the extrapolated Landweber operator $\mathcal{L}_{\sigma }\{T\}$
with $\sigma $ bounded from above by $\tau $ defined by
\begin{equation}
\tau (x):=%
\begin{cases}
\displaystyle\left( \frac{\Vert A\Vert \cdot \Vert T(Ax)-Ax\Vert }{\Vert
A^{\ast }(T(Ax)-Ax)\Vert }\right) ^{2}\text{,} & \text{if }Ax\notin \limfunc{%
Fix}T\text{,} \\
1\text{,} & \text{if }Ax\in \limfunc{Fix}T\text{.}%
\end{cases}
\label{e-tau}
\end{equation}%
By (\ref{e-FixT}), $\tau (x)$ and $\mathcal{L}_{\tau }\{T\}x$ are both well
defined. Moreover, it is not difficult to see that $\tau (x)\geq 1$, since
\begin{equation}
\Vert A^{\ast }(T(Ax)-Ax)\Vert \leq \Vert A\Vert \cdot \Vert T(Ax)-Ax\Vert
\end{equation}%
and thus $\mathcal{L}_{\sigma }\{T\}x$ is also well defined.

Observe that when $\sigma =\tau $, then  we have
\begin{equation}
\mathcal{L}_{\tau }\{T\}x=%
\begin{cases}
\displaystyle x+\frac{\Vert T(Ax)-Ax\Vert ^{2}}{\Vert A^{\ast
}(T(Ax)-Ax)\Vert ^{2}}A^{\ast }(T(Ax)-Ax))\text{,} & \text{if }Ax\notin
\limfunc{Fix}T\text{,} \\
x\text{,} & \text{if }Ax\in \limfunc{Fix}T\text{,}%
\end{cases}
\label{e-extLand}
\end{equation}%
that is, $\mathcal{L}_{\tau }\{T\}x$ does not depend on $\Vert A\Vert $.
Similarly to Lemma \ref{l-FixLT}, we have the following result.

\begin{lemma}
\label{l-extLand-SQNE} Let $\mathcal{L}\{T\}$ be the Landweber operator  corresponding  to $T:\mathcal{H}_{2}\rightarrow \mathcal{H}_{2}$, $\lambda \in
(0,1] $ and $\rho \geq 0$. Let $\sigma \colon \mathcal{H}_1\to [1,\infty)$
be an extrapolation function bounded from above by $\tau$  as defined in %
\eqref{e-tau}. If $T$ is $\rho $-SQNE and $\limfunc{im}A\cap \limfunc{Fix}%
T\neq \emptyset $, then $\mathcal{L}_{\lambda\sigma}\{T\}:=\limfunc{Id}%
+\lambda\sigma(\cdot)(\mathcal{L}_\sigma\{T\}-\limfunc{Id})$ is also $\rho$%
-SQNE and we have $\limfunc{Fix}\mathcal{L}_{\lambda\sigma}\{T\}=A^{-1}(%
\limfunc{Fix}T)$.
\end{lemma}

\begin{proof}
The operator $\mathcal{L}_{\tau }\{T\}$ is $\rho$-SQNE by \cite[Theorem 4.1]%
{CM16}. Observe that by defining $\alpha(x):=\lambda\frac{\sigma(x)}{\tau(x)}%
\in(0,1]$, we have $\mathcal{L}_{\lambda\sigma }\{T\}x =x+\alpha(x)(\mathcal{%
L}_\tau\{T\}x-x), $ that is, $\mathcal{L}_{\lambda\sigma }\{T\}$ is  an  $\alpha(\cdot)$-relaxation of $\mathcal{L}_{\tau }\{T\}$. Thus the result
follows by \eqref{eq:SQNEequiv:ineq2} from Lemma \ref{th:SQNEequiv}.
\end{proof}

\bigskip Observe that for any  point  $x\in \mathcal{H}_{1}$, we have
\begin{equation}
\Vert \mathcal{L}_{\sigma }\{T\}x-x\Vert =\sigma (x)\Vert \mathcal{L}%
\{T\}x-x\Vert \geq \Vert \mathcal{L}\{T\}x-x\Vert .  \label{e-UV}
\end{equation}%
This, when combined with either Theorem \ref{t-regLand} or Corollary \ref%
{c-BoundRegLand}, leads to the following two results.

\begin{theorem}
\label{t-regExtrLand} Let $A:\mathcal{H}_{1}\rightarrow \mathcal{H}_{2}$ be
a nonzero bounded linear operator, let $T:\mathcal{H}_{2}\rightarrow
\mathcal{H}_{2}$ be $\rho $-SQNE, where $\rho \geq 0$ and let $S\subseteq
\mathcal{H}_{2}$ be nonempty. Assume that $\limfunc{im}A\cap \limfunc{Fix}%
T\neq \emptyset $. Let $\mathcal{L}_{\sigma }\{T\}:\mathcal{H}%
_{1}\rightarrow \mathcal{H}_{1}$ be the extrapolated Landweber operator with
$\sigma =\tau $, where $\tau $ is defined by \eqref{e-tau}. Then the following statements hold:

\begin{enumerate}
\item[$\mathrm{(i)}$] If $T$ is weakly regular over $S$, then $\mathcal{L}%
_\sigma \{T\}$ is weakly regular over $A^{-1}(S)$.

\item[$\mathrm{(ii)}$] If $\limfunc{im}A$ is closed, $T$ is regular over $S$%
, $S$ is bounded and $\{\limfunc{im}A,\limfunc{Fix}T\}$ is regular over $S$,
then $\mathcal{L}_{\sigma }\{T\}$ is regular over $A^{-1}(S)$.

\item[$\mathrm{(iii)}$] If $\limfunc{im}A$ is closed, $T$ is linearly
regular with modulus $\delta >0$ over $S$ and the family $\{\limfunc{im}A,
\limfunc{Fix}T\}$ is linearly regular with modulus $\kappa >0$ over $S$,
then $\mathcal{L}_{\sigma }\{T\}$ is linearly regular over $A^{-1}(S)$ with
modulus $\Delta $ defined by \eqref{e-Delta}. Moreover, for any $x\in
A^{-1}(S)$, we have
\begin{equation}
\Vert \mathcal{L}_{\sigma }\{T\}x-x\Vert \geq \sigma (x)\frac{\rho +1}{2}%
\left( \frac{\delta |A|}{\kappa \Vert A\Vert }\right) ^{2}d(x,\limfunc{Fix}%
\mathcal{L}_{\sigma }\{T\})\text{.}  \label{e-DeltaInequality}
\end{equation}
\end{enumerate}
\end{theorem}

\begin{corollary}
Let $A:\mathcal{H}_{1}\rightarrow \mathcal{H}_{2}$ be a nonzero bounded
linear operator and let $T:\mathcal{H}_{2}\rightarrow \mathcal{H}_{2}$ be $%
\rho $-SQNE, where $\rho \geq 0$. Assume that $\limfunc{im}A\cap \limfunc{Fix%
}T\neq \emptyset $. Let $\mathcal{L}_{\sigma }\{T\}:\mathcal{H}%
_{1}\rightarrow \mathcal{H}_{1}$ be the extrapolated Landweber operator with
$\sigma =\tau $, where $\tau $ is defined by \eqref{e-tau}. Then the following statements hold:

\begin{enumerate}
\item[$\mathrm{(i)}$] If $T$ is weakly regular, then $\mathcal{L}_{\sigma
}\{T\}$ is weakly regular.

\item[$\mathrm{(ii)}$] If $\limfunc{im}A$ is closed, $T$ is boundedly
regular and the family $\{\limfunc{im}A,\limfunc{Fix}T\}$ is boundedly regular, then $\mathcal{L}_{\sigma }\{T\}$ is boundedly regular.

\item[$\mathrm{(iii)}$] If $\limfunc{im}A$ is closed, $T$ is boundedly
linearly regular and the family $\{\limfunc{im}A,\limfunc{Fix}T\}$ is
boundedly linearly regular, then  $\mathcal{L}_\sigma\{T\}$  is boundedly linearly
regular.
\end{enumerate}
\end{corollary}

\begin{remark}
\rm\
As in Remark \ref{r-LandRelCutter}, assume that $T$ is a cutter ($\rho =1$)
and that all the assumptions of Theorem \ref{t-regExtrLand}(iii) are
satisfied. Let $\lambda \in (0,2]$. Then the relaxed Landweber operator $%
\mathcal{L}_{\lambda \sigma }\{T\}$ defined by
\begin{equation}
\mathcal{L}_{\lambda \sigma }\{T\}x :=x+\lambda \sigma (x)\big(\mathcal{L}%
\{T\}x-x\big), \quad  x\in\mathcal H_1,
\end{equation}%
is $\frac{2-\lambda }{\lambda }$-SQNE and satisfies
\begin{equation}
\Vert \mathcal{L}_{\lambda \sigma }\{T\}x-x\Vert =\lambda \sigma (x)\Vert
\mathcal{L}\{T\}x-x\Vert \geq \lambda \sigma (x)\left( \frac{\delta |A|}{%
\kappa \Vert A\Vert }\right) ^{2}d(x,\limfunc{Fix}\mathcal{L}_{\lambda
\sigma }\{T\})\text{,}
\end{equation}%
where $x\in S$. As in Remark \ref{r-LandRelCutter}, one can adjust the above
inequality for $T=P_{Q}$, where $Q\subseteq \mathcal{H}_{2}$ is  closed and
convex, and $Q\cap \limfunc{im}A\neq \emptyset .$
\end{remark}

\section{Applications to the Split Convex Feasibility Problem\label{s-SCFP}}

In this section we propose a few $CQ$-type methods for solving the SCFP defined by \eqref{int:SCFP}-\eqref{int:SCFP-fix}, that is,
\begin{equation}
\text{Find }x\in C=\limfunc{Fix}S\text{  such that }Ax\in Q=\limfunc{Fix}T,
\end{equation}%
where $S,T$ are given operators.

\begin{theorem}
\label{t-ProjLand}Let $S:\mathcal{H}_{1}\rightarrow \mathcal{H}_{1}$ and $T:\mathcal{H}_{2} \rightarrow \mathcal{H}_{2}$ be $\rho _{S}$- and $\rho _{T}$-SQNE, respectively, where $\rho _{S},\rho _{T}>0$. Let the sequence  $\{x_{k}\}_{k=0}^\infty$  be defined by the method
\begin{equation}
x_{0}\in \mathcal{H}_{1};\qquad x_{k+1}:=S\left( x_{k}+\lambda _{k}\frac{%
\sigma (x_{k})}{\Vert A\Vert ^{2}}A^{\ast }\Big(T(Ax_{k})-Ax_{k}\Big)\right)
,  \label{e-ProjLand}
\end{equation}%
where the relaxation parameters satisfy $\lambda _{k}\in \lbrack \varepsilon
,1]$ for some $\varepsilon \in (0,1)$ and the extrapolation function $\sigma
:\mathcal{H}_{1}\rightarrow \lbrack 1,\infty )$ is bounded from above by $\tau $  as  defined in \eqref{e-tau} (in particular, one can use $\sigma (x):=1$ or $\sigma (x):=\tau (x)$ for all $x$). Assume that the solution set $F:=%
\limfunc{Fix}S\cap A^{-1}(\limfunc{Fix}T)\neq \emptyset $. Then the
following statements hold:

\begin{enumerate}
\item[$\mathrm{(i)}$] If $S$ and $T$ are both weakly regular, then  the sequence $\{x_{k}\}_{k=0}^\infty$  weakly converges to some $x_{\infty }\in F$.

\item[$\mathrm{(ii)}$] If $\limfunc{im}A$ is closed, $S$ and $T$ are both
boundedly regular, and the two families of sets $\{\limfunc{im}A,\limfunc{Fix%
}T\}$ and $\{\limfunc{Fix}S,A^{-1}(\limfunc{Fix}T)\}$ are boundedly regular,
then the convergence to $x_{\infty }$ is in norm.

\item[$\mathrm{(iii)}$] If $\limfunc{im}A$ is closed, $S$ and $T$ are both
boundedly linearly regular, and the following two families of sets $\{%
\limfunc{im}A,\limfunc{Fix}T\}$ and $\{\limfunc{Fix}S,A^{-1}(\limfunc{Fix}%
T)\}$ are boundedly linearly regular, then the convergence to $x_{\infty }$
is at least linear, that is,
\begin{equation}
d(x_{k+1},F)\leq qd(x_{k},F)\quad \text{and}\quad \Vert x_{k}-x_{\infty
}\Vert \leq 2d(x_{0},F)q^{k},  \label{e-ProjLandRate}
\end{equation}%
for some $q\in (0,1)$, which may depend on $x_{0}$.
\end{enumerate}
\end{theorem}

\begin{proof}
Observe that since we can write $x_{k+1}=S\mathcal{L}_{\lambda _{k}\sigma
}\{T\}x_k $, the sequence  $\{x_{k}\}_{k=0}^\infty$  is Fej\'{e}r monotone with respect to $F$.
Indeed, by Lemma \ref{l-extLand-SQNE}, the operator $\mathcal{L}_{\lambda
\sigma }\{T\}$ is $\rho _{T}$-SQNE and $\limfunc{Fix}\mathcal{L}_{\lambda
\sigma }\{T\}=A^{-1}(\limfunc{Fix}T)$. By Theorem \ref{t-SQNE}, the product $%
S\mathcal{L}_{\lambda _{k}\sigma }\{T\}$ is $(\frac{1}{2}\min \{\rho
_{S},\rho _{T}\})$-SQNE with $\limfunc{Fix}S\mathcal{L}_{\lambda _{k}\sigma
}\{T\}=F$. Hence for any  point  $z\in F$, we have
\begin{equation}
\Vert x_{k+1}-z\Vert ^{2}\leq \Vert x_{k}-z\Vert ^{2}-\frac{\min \{\rho
_{S},\rho _{T}\}}{2}\Vert x_{k+1}-x_{k}\Vert ^{2}.  \label{pr-ProjLandFM}
\end{equation}%
Moreover, since  $\{\Vert x_{k}-z\Vert \}_{k=0}^\infty$  converges as a decreasing sequence, we have
\begin{equation}
\left\Vert S\mathcal{L}_{\lambda _{k}\sigma }x_{k}-x_{k}\right\Vert =\Vert
x_{k+1}-x_{k}\Vert \rightarrow 0.  \label{pr-ProjLandXk}
\end{equation}%
By Fej\'{e}r monotonicity, we  also   have
\begin{equation}
\{x_{k}\}\subset B_{1}:=\{x\in \mathcal{H}_{1}\mid \Vert x-P_{F}x_{0}\Vert
\leq r:=d(x_{0},F)\}.
\end{equation}%
Consequently,
\begin{equation}
\{Ax_{k}\}\subset B_{2}:=\{y\in \mathcal{H}_{2}\mid \Vert y-AP_{F}x_{0}\Vert
\leq r\Vert A\Vert \}
\end{equation}%
and
\begin{equation}
B_{1}\subseteq A^{-1}(B_{2}).  \label{pr-ProjLandB1B2}
\end{equation}%
The remaining part of the proof follows from Theorems \ref{t-Fejer}, \ref%
{t-regOper} and \ref{t-regExtrLand}.

\textit{Part (i).} By assumption, $S$ is weakly regular over $B_{1}$ and $T$
is weakly regular over $B_{2}$. By Theorem \ref{t-regExtrLand} (i) and by \eqref{pr-ProjLandB1B2}, the operator $\mathcal{L}_{\sigma }\{T\}$ is weakly
regular over $B_{1}$. It is not difficult to see that the sequence of
relaxations  $\{\mathcal{L}_{\lambda _{k}\sigma }\}_{k=0}^\infty$  is also weakly regular over $B_{1}$ (see \cite[Proposition 4.7]{CRZ18}). Hence, by Theorem \ref{t-regOper} (i), the product sequence  $\{S\mathcal{L}_{\lambda _{k}\sigma }\}_{k=0}^\infty$  is weakly regular over $B_{1}$ as well. Thus (by the definition of weak regularity) any weak cluster point of  $\{x_{k}\}_{k=0}^\infty$  is in $F$, which shows that  $\{x_{k}\}_{k=0}^\infty$  converges weakly to some point $x_{\infty }\in F$ in view of Theorem \ref{t-Fejer} (i).

\textit{Part (ii).}  Using an argument similar to the one above,  we conclude that  $\{\mathcal{L}_{\lambda _{k}\sigma }\}_{k=0}^\infty$  is regular over $B_{1}$. By applying Theorem \ref{t-regOper} (ii), we see that  $\{S\mathcal{L}_{\lambda _{k}\sigma }\}_{k=0}^\infty$  is also regular over $B_{1}$. Hence by \eqref{pr-ProjLandXk}, when combined with the definition of regularity over $B_{1}$, we see that $d(x_{k},F)\rightarrow 0$, which implies that $\Vert x_{k}-x_{\infty }\Vert \rightarrow 0$ in view of Theorem \ref{t-Fejer} (ii).

\textit{Part (iii).} By assumption, the operator $T$ and the family $\{%
\limfunc{im}A, \limfunc{Fix}T\}$ are both linearly regular over $B_2$ with
moduli $\delta_T$ and $\kappa_2$, respectively. Thus, by Theorem \ref%
{t-regExtrLand} (iii) and \eqref{pr-ProjLandB1B2}, the operator $\mathcal{L}%
_{\sigma}\{T\}$ is linearly regular over $B_1$ with modulus $\Delta$ defined
in \eqref{e-Delta} (with $\kappa:=\kappa_2$). Consequently, the sequence  $\{\mathcal{L}_{\lambda_k\sigma}\{T\}\}_{k=0}^\infty$  is linearly regular over the same ball with modulus $\varepsilon\Delta$.

By another assumption, the family $\{\limfunc{Fix}S, A^{-1}(\limfunc{Fix}%
T)\} $ is linearly regular over $B_1$ with modulus $\kappa_1$. By applying
Theorem \ref{t-regOper} (iii) to  $\{S\}_{k=0}^\infty$ and $\{\mathcal{L}_{\lambda_k\sigma}\{T\}\}_{k=0}^\infty$,  we conclude that the product sequence  $\{S\mathcal{L}_{\lambda_k\sigma}\}_{k=0}^\infty$  is also linearly regular over $B_1$ with
modulus
\begin{equation}
\Gamma= \min\{\rho_S, \rho_T\} \left(\frac{\min\{\delta_S,\varepsilon
\Delta\}}{2\kappa_1}\right)^2,
\end{equation}
that is, $\|x_{k+1}-x_k\|\geq \Gamma d(x_k,F)$. Hence, by setting $z:=P_Fx_k$
in \eqref{pr-ProjLandFM} and by the inequality $d(x_{k+1},F)\leq
\|x_{k+1}-P_Fx_k\|$, we arrive at
\begin{equation}
d^2(x_{k+1}, F)\leq d^2(x_k,F) - \frac{\min\{\rho_S, \rho_T\}}{2} \Gamma^2
d^2(x_{k}, F).
\end{equation}
Consequently, $d(x_{k+1})\leq q d(x_k, F)$ with
\begin{equation}
q:=\sqrt{1- \frac{\min\{\rho_S, \rho_T\}}{2}\Gamma^2}
\end{equation}
and, by Theorem \ref{t-Fejer}(iii), we also have $\|x_k-x_\infty\|\leq
2d(x_0,F)q^k$. This completes the proof.
\end{proof}

\bigskip

Note again, as in (\ref{e-extLand}), that applying $\sigma (x)=\tau (x)$ in (\ref{e-ProjLand}) with $\tau (x)$ defined by (\ref{e-tau}), we do not need to know the norm of $A$.

\begin{remark}[Cutters $S$ and $T$]
\rm\ %
\label{r-ProjLandCutters}\textrm{If we assume that both $S$ and $T$ are
cutters, then the relaxation parameters $\lambda _{k}$ can be chosen from
the interval $[\varepsilon ,2-\varepsilon ]$ for some $\varepsilon \in (0,1)$. Indeed, define $U:=\limfunc{Id}+(2-\varepsilon )(T-\limfunc{Id})$ and  $\alpha_{k}:= \frac{\lambda _{k}}{2-\varepsilon}$.  Then $U$ is $\frac{%
\varepsilon }{2-\varepsilon }$-SQNE, $\alpha _{k}\in \lbrack \frac{%
\varepsilon }{2-\varepsilon },1]$ and $S\mathcal{L}_{\lambda _{k}\sigma
}\{T\}x=S\mathcal{L}_{\alpha _{k}\sigma }\{U\}x$. Moreover, the upper bound
for $\sigma $ defined by \eqref{e-tau} is exactly the same as the
corresponding upper bound determined by $U$, that is, for $Ax\notin \limfunc{%
Fix}T=\limfunc{Fix}U$, we have
\begin{equation}
\tau (x)=\left( \frac{\Vert A\Vert \cdot \Vert T(Ax)-Ax\Vert }{\Vert A^{\ast
}(T(Ax)-Ax)\Vert }\right) ^{2}=\left( \frac{\Vert A\Vert \cdot \Vert
U(Ax)-Ax\Vert }{\Vert A^{\ast }(U(Ax)-Ax)\Vert }\right) ^{2}.
\end{equation}%
Consequently, the sequence  defined by
\begin{equation}
x_{0}\in \mathcal{H}_{1};\qquad x_{k+1}:=S\mathcal{L}_{\lambda _{k}\sigma
}\{T\}x_{k}
\end{equation}%
is a particular case of the iteration defined in Theorem \ref{t-ProjLand}.
In addition, the convergence statements (i), (ii) and (iii) hold here as
well since weak/bounded/bounded linear regularity of $T$ implies the same
type of regularity for $U$. }
\end{remark}

\begin{corollary}[Extrapolated $CQ$-method]
\label{c-CQmethod} Let $C\subseteq \mathcal{H}_{1}$ and $Q\subseteq \mathcal{%
H}_{2}$ be nonempty, closed and convex. Let the sequence  $\{x_{k}\}_{k=0}^\infty$  be defined by the method
\begin{equation}
x_{0}\in \mathcal{H}_{1};\qquad x_{k+1}:=P_{C}\left( x_{k}+\lambda _{k}\frac{%
\sigma (x_{k})}{\Vert A\Vert ^{2}}A^{\ast }\Big(P_{Q}(Ax_{k})-Ax_{k}\Big)%
\right) ,
\end{equation}%
where the relaxation parameters satisfy $\lambda _{k}\in \lbrack \varepsilon
,2-\varepsilon ]$ for some $\varepsilon \in (0,1)$ and the extrapolation
function $\sigma \colon \mathcal{H}_{1}\rightarrow \lbrack 1,\infty )$ is
bounded from above by $\tau $ defined by (compare with \eqref{e-tau})
\begin{equation}
\tau (x):=%
\begin{cases}
\displaystyle\left( \frac{\Vert A\Vert \cdot \Vert P_{Q}(Ax)-Ax\Vert }{\Vert
A^{\ast }(P_{Q}(Ax)-Ax)\Vert }\right) ^{2}\text{,} & \text{if }Ax\notin Q%
\text{,} \\
1\text{,} & \text{if }Ax\in Q\text{;}%
\end{cases}
\label{e-tauCQ}
\end{equation}%
(in particular, one can use $\sigma (x):=1$ or $\sigma (x):=\tau (x)$ for
all $x$). Assume that $F:=C\cap A^{-1}(Q)\neq \emptyset $. Then the
following statements hold:

\begin{enumerate}
\item[$\mathrm{(i)}$]  The sequence  $\{x_{k}\}_{k=0}^\infty$  weakly converges to some $x_{\infty }\in F$.

\item[$\mathrm{(ii)}$] If $\limfunc{im}A$ is closed, and the two families of
sets $\{\limfunc{im}A,Q\}$ and $\{C,A^{-1}(Q)\}$ are boundedly regular, then
the convergence to $x_{\infty }$ is in norm.

\item[$\mathrm{(iii)}$] If $\limfunc{im}A$ is closed, and the two families
of sets $\{\limfunc{im}A,Q\}$ and $\{C,A^{-1}(Q)\}$ are boundedly linearly
regular, then the convergence to $x_{\infty }$ is at least linear, that is,
\begin{equation}
d(x_{k+1},F)\leq qd(x_{k},F)\quad \text{and}\quad \Vert x_{k}-x_{\infty
}\Vert \leq 2d(x_{0},F)q^{k},  \label{e-ProjLandRate}
\end{equation}%
for some $q\in (0,1)$, which may depend on $x_{0}$.
\end{enumerate}
\end{corollary}

\begin{proof}
The projections $P_{C}$ and $P_{Q}$ are linearly regular cutters. Recall that
linear regularity implies regularity and this in turn implies weak
regularity; see \cite[Corollary 4.4]{CRZ18}. Hence, the result follows  from  Theorem \ref{t-ProjLand} and Remark \ref{r-ProjLandCutters}.
\end{proof}

\begin{example}
\rm\ %
\textrm{Assume that }$\mathrm{\limfunc{im}}$\textrm{$A$ is closed. Then the two
additional regularity conditions mentioned in Corollary \ref{c-CQmethod}
(iii) are satisfied if, for example, $C\cap \limfunc{int}A^{-1}(Q)\neq
\emptyset $, or when $A(C)\cap \limfunc{int}Q\neq \emptyset $. This
follows from Example \ref{e-RegSets} (iii). }
\end{example}

\begin{corollary}[Extrapolated-subgradient $CQ$ method]
\label{c-CQmethodSub} Let $C=\{x\in \mathcal{H}_{1}\mid c(x)\leq 0\}\neq
\emptyset $ and $Q=\{y\in \mathcal{H}_{2}\mid q(x)\leq 0\}\neq \emptyset $
for some lower semi-continuous and convex functions $c$ and $q$. Let the sequence  $\{x_{k}\}_{k=0}^\infty$  be defined by the method
\begin{equation}
x_{0}\in \mathcal{H}_{1};\qquad x_{k+1}:=P_{c}\left( x_{k}+\lambda _{k}\frac{%
\sigma (x_{k})}{\Vert A\Vert ^{2}}A^{\ast }\Big(P_{q}(Ax_{k})-Ax_{k}\Big)%
\right) ,
\end{equation}%
where $P_{c}$ and $P_{q}$ are subgradient projections (see Example \ref%
{e-subProj}), the relaxation parameters satisfy $\lambda _{k}\in \lbrack
\varepsilon ,2-\varepsilon ]$ for some $\varepsilon \in (0,1)$ and the
extrapolation function $\sigma \colon \mathcal{H}_{1}\rightarrow \lbrack
1,\infty )$ is bounded from above by $\tau $ defined by (compare with %
\eqref{e-tau})
\begin{equation}
\tau (x):=%
\begin{cases}
\displaystyle\left( \frac{\Vert A\Vert \cdot \Vert P_{q}(Ax)-Ax\Vert }{\Vert
A^{\ast }(P_{q}(Ax)-Ax)\Vert }\right) ^{2}\text{,} & \text{if }q(Ax)>0\text{,%
} \\
1\text{,} & \text{if }q(Ax)\leq 0;%
\end{cases}
\label{e-tauCQsub}
\end{equation}%
(in particular, one can use $\sigma (x):=1$ or $\sigma (x):=\tau (x)$ for
all $x$). Assume that $F:=C\cap A^{-1}(Q)\neq \emptyset $, and that $c$ and $%
q $ are Lipschitz continuous on bounded sets (see Example \ref{e-subProjReg}
and Remark \ref{r-subdifferential}). Then the following statements hold:

\begin{enumerate}
\item[$\mathrm{(i)}$]  The sequence  $\{x_{k}\}_{k=0}^\infty$  weakly converges to some $x_{\infty }\in F$.

\item[$\mathrm{(ii)}$] If $\limfunc{im}A$ is closed, $P_{c}$ and $P_{q}$ are
both boundedly regular (for example, when both functions $c$ and $q$ are
strongly convex) and the two families of sets $\{\limfunc{im}A,Q\}$ and $%
\{C,A^{-1}(Q)\}$ are boundedly regular, then the convergence to $x_{\infty }$
is in norm.

\item[$\mathrm{(iii)}$] If $\limfunc{im}A$ is closed, $P_{c}$ and $P_{q}$
are both boundedly linearly regular and the two families of sets $\{\limfunc{%
im}A,Q\}$ and $\{C,A^{-1}(Q)\}$ are boundedly linearly regular, then the
convergence to $x_{\infty }$ is at least linear, that is,
\begin{equation}
d(x_{k+1},F)\leq qd(x_{k},F)\quad \text{and}\quad \Vert x_{k}-x_{\infty
}\Vert \leq 2d(x_{0},F)q^{k}  \label{e-CQsubRate}
\end{equation}%
for some $q\in (0,1)$ which may depend on $x_{0}$.
\end{enumerate}
\end{corollary}

\begin{example}
\rm\ %
\textrm{All the regularity conditions mentioned in Corollary \ref%
{c-CQmethodSub} (iii) are satisfied if, for example, there is $z\in \mathcal{%
H}_{1}$ such that $c(z)<0$ and $q(Az)<0$. This follows from Examples \ref%
{e-RegSets}(iii) and \ref{e-subProjReg}(iii). }
\end{example}

\begin{remark}[Bounded linear regularity of the SCFP]
\rm\ %
\label{r-BLRofSCFP}\textrm{\ Conditions presented in Theorem \ref{t-ProjLand}
(iii) and in Corollary \ref{c-CQmethod} (iii) imply that the split convex feasibility
problem  has  the bounded linear regularity  property  in the sense of \cite[Definition 2.2]{WHLY17}; compare with \eqref{int:BLRofSCFP}. Indeed, by assumption, $\{C,A^{-1}(Q)\}$ is $\kappa _{1}$-linearly regular over $B_{1}:=\{x\in \mathcal{H}_{1}\mid \Vert x\Vert \leq r\}$ and $\{\limfunc{im}A,Q\}$ is $\kappa _{2}$-linearly regular over $B_{2}:=\{y\in \mathcal{H}_{2}\mid \Vert y\Vert \leq \Vert A\Vert r\}$. Consequently, for any $x\in C\cap B_{1}$, we have $Ax\in B_{2}$ and thus, by Lemma \ref{l-A1},
\begin{equation}
d(Ax,Q)\geq \frac{1}{\kappa _{2}}d(Ax,\limfunc{im}A\cap Q)\geq \frac{|A|}{%
\kappa _{2}}d\big(x,A^{-1}(Q)\big)\geq \frac{|A|}{\kappa _{1}\kappa _{2}}d%
\big(x,C\cap A^{-1}(Q)\big).
\end{equation}%
 At this point, it  is worth emphasizing that Theorem \ref{t-ProjLand} also applies to operators other than projections. }
\end{remark}

\bigskip

\textbf{Acknowledgement.} This research was supported in part by the Israel
Science Foundation (Grants no. 389/12 and 820/17), by the Fund for the
Promotion of Research at the Technion and by the Technion General Research
Fund.

\end{document}